\documentclass[oneside,english]{amsart}
\usepackage[T1]{fontenc}
\usepackage[latin9]{inputenc}
\usepackage{amsthm,amssymb,youngtab,verbatim,fullpage}
\makeatletter
\numberwithin{equation}{section}
\numberwithin{figure}{section}
\theoremstyle{plain}
\newtheorem{thm}{\protect\theoremname}[section]
  \theoremstyle{definition}
  \newtheorem{defn}[thm]{\protect\definitionname}
  \theoremstyle{definition}
  \newtheorem{example}[thm]{\protect\examplename}
  \theoremstyle{plain}
  \newtheorem{lem}[thm]{\protect\lemmaname}
  \theoremstyle{plain}
  \newtheorem{prop}[thm]{\protect\propositionname}
  \theoremstyle{plain}
  \newtheorem{claim}[thm]{\protect\claimname}
  \theoremstyle{plain}
  \newtheorem{observation}[thm]{\protect\observationname}
  \theoremstyle{plain}
  \newtheorem{cor}[thm]{\protect\corollaryname}
  \theoremstyle{plain}
  \newtheorem{conj}[thm]{\protect\conjecturename}
   \theoremstyle{plain}
  \newtheorem{remark}[thm]{\protect\remarkname}

\makeatother

\usepackage{babel}
  \providecommand{\definitionname}{Definition}
  \providecommand{\examplename}{Example}
  \providecommand{\lemmaname}{Lemma}
  \providecommand{\propositionname}{Proposition}
   \providecommand{\corollaryname}{Corollary}
   \providecommand{\theoremname}{Theorem}
   \providecommand{\claimname}{Claim}
  \providecommand{\observationname}{Observation}
  \providecommand{\conjecturename}{Conjecture}
  \providecommand{\remarkname}{Remark}
  
\DeclareMathOperator{\Domain}{\textrm{Domain}}
\DeclareMathOperator{\Range}{\textrm{Range}}
\begin{document}
\global\long\def\connected{\text{highly connected}}
\global\long\def\f{\mathcal{F}}
\global\long\def\pn{\mathcal{P}\left(\left[n\right]\right)}
\global\long\def\g{\mathcal{G}}
\global\long\def\l{\mathcal{L}}
\global\long\def\s{\mathcal{S}}
\global\long\def\j{\mathcal{J}}
\global\long\def\d{\mathcal{D}}
\global\long\def\Cay{\mathrm{Cay}}
\global\long\def\GL{\mathrm{GL}}
\global\long\def\Inf{}
\global\long\def\Id{\textrm{Id}}
\global\long\def\Tr{\mathrm{Tr}}
\global\long\def\sgn{\textrm{sgn}}
\global\long\def\p{\mathcal{P}}
\global\long\def\h{\mathcal{H}}
\global\long\def\n{\mathbb{N}}
\global\long\def\a{\mathcal{A}}
\global\long\def\b{\mathcal{B}}
\global\long\def\c{\mathcal{C}}
\global\long\def\e{\mathbb{E}}
\global\long\def\x{\mathbf{x}}
\global\long\def\y{\mathbf{y}}
\global\long\def\z{\mathbf{z}}
\global\long\def\c{\mathbf{c}}
\global\long\def\av{\mathsf{A}}
\global\long\def\chop{\mathrm{Chop}}
\global\long\def\stab{\mathrm{Stab}}
\global\long\def\Span{\mathrm{Span}}
\global\long\def\codim{\mathrm{codim}}
\global\long\def\Var{\mathrm{Var}}
\global\long\def\rank{\mathrm{rank}}
\global\long\def\t{\mathsf{T}}

\title[Juntas in the symmetric group, and forbidden intersection problems]{Approximation by juntas in the symmetric group,\\and forbidden intersection problems}
\author{David Ellis}
\address{School of Mathematics, University of Bristol, Fry Building, Woodland Road, Bristol, BS8 1UG, United Kingdom.}
\email{david.ellis@bristol.ac.uk}
\author{Noam Lifshitz}
\address{Einstein Institute of Mathematics, Hebrew University of Jerusalem, Edmond J. Safra Campus, Givat Ram, Jerusalem 91904, Israel.}
\email{noamlifshitz@gmail.com}

\date{December 2019}

\begin{abstract}
A family of permutations $\f\subset S_{n}$ is said to be \emph{$t$-intersecting} if any two permutations in $\f$ agree on at least $t$ points. It is said to be \emph{$(t-1)$-intersection-free} if no two permutations in $\f$ agree on exactly $t-1$ points. If $S,T \subset \{1,2,\ldots,n\}$ with $|S|=|T|$, and $\pi: S \to T$ is a bijection, the \emph{$\pi$-star in $S_n$}
is the family of all permutations in $S_n$ that agree with $\pi$ on all of $S$. An {\em $s$-star} is a $\pi$-star such that $\pi$ is a bijection between sets of size $s$.

Friedgut and Pilpel, and independently the first author (see \cite{efp}), showed that if $\f \subset S_n$ is $t$-intersecting, and $n$ is sufficiently large depending on $t$, then $|\f| \leq (n-t)!$; this proved a conjecture of Deza and Frankl from 1977. Equality holds only if $\f$ is a $t$-star.

The Erd\H{o}s-S\'os forbidden intersection problem (for uniform hypergraphs) has been intensively studied; it has been resolved in many cases but is still open in full generality. In this paper, we consider the forbidden intersection problem for families of permutations. We prove a considerable strengthening of the Deza-Frankl conjecture, namely that if $n$ is sufficiently large depending on $t$, and $\f \subset S_n$ is $(t-1)$-intersection-free, then $|\f| \leq (n-t)!$, with equality only if $\f$ is a $t$-star.

The proof is rather different to the original proof of the Deza-Frankl conjecture, and is quite a lot more robust. The main ingredient is a `junta approximation' result, namely, that any $(t-1)$-intersection-free family of permutations is essentially contained in a $t$-intersecting {\em junta} (a `junta' being a union of a bounded number of $O(1)$-stars). The proof of our junta approximation result relies, in turn, on (i) a weak regularity lemma for families of permutations (which outputs a `junta' whose stars are intersected by $\f$ in a weakly pseudorandom way), (ii) a combinatorial argument that `bootstraps' the weak notion of pseudorandomness into a stronger one, and finally (iii) a spectral argument for pairs of highly-pseudorandom fractional families (this spectral argument being significantly shorter than the spectral argument in \cite{efp}, though still non-trivial). Our proof employs four different notions of pseudorandomness, three being combinatorial in nature, and one being algebraic. The connection we demonstrate between these combinatorial and algebraic notions of psuedorandomness may find further applications.

We also use our junta approximation result to prove some new stability results on $t$-intersecting families and $(t-1)$-intersection-free families.
\end{abstract}

\maketitle

\section{Introduction}
Erd\H{o}s-Ko-Rado type problems are an important class of problems within Extremal Combinatorics. In general, an Erd\H{o}s-Ko-Rado type problem asks for the maximum possible size of a family of objects, subject to some intersection condition on pairs of objects in the family. For example, we say a family of sets is {\em intersecting} if any two sets in the family have nonempty intersection. The classical Erd\H{o}s-Ko-Rado theorem \cite{ekr} states that if $k < n/2$, an intersecting family of $k$-element subsets of $\{1,2,\ldots,n\}$ has size at most ${n-1 \choose k-1}$, and that if equality holds, then the family must consist of all $k$-element subsets containing some fixed element. Over the last sixty years, many other Erd\H{o}s-Ko-Rado type results have been obtained, for different mathematical structures (e.g.\ for families of graphs \cite{cfgs,triangle}, and families of partitions \cite{meagher}) and under different intersection conditions on the sets in the family. We mention in particular the seminal theorem of Ahlswede and Khachatrian \cite{ak} which specifies, for each $(n,k,t) \in \mathbb{N}^3$, the largest possible size of a $t$-intersecting family of $k$-element subsets of $\{1,2,\ldots,n\}$. (We say a family of sets is {\em $t$-intersecting} if any two sets in the family have intersection of size at least $t$.) As well as being natural in their own right, Erd\H{o}s-Ko-Rado type questions have found applications in Computer Science and Coding Theory, and the techniques developed in solving them have found wide applicability in many other areas of Mathematics. The reader is referred to \cite{DF83,MV15} for surveys of this area of research and its applications; in fact, the  latter deals with the broader subject of {\em Tur\'an-type} problems, which ask for the maximum possible size of some structure, subject to a certain substructure being forbidden.

One particularly challenging type of Erd\H{o}s-Ko-Rado problem concerns what happens when just one intersection-size is forbidden. The {\em Erd\H{o}s-S\'os forbidden intersection problem} \cite{erdos-sos} is to determine, for each $(n,k,t) \in \mathbb{N}^3$, the maximum possible size of a family of $k$-element subsets of $\{1,2,\ldots,n\}$ such that no two sets in the family have intersection of size exactly $t-1$. This problem remains open in full generality, unlike the $t$-intersection problem above (solved by Ahlswede and Khachatrian), though it has been solved for quite a wide range of the parameters, by Frankl and F\"uredi \cite{ff}, Keevash, Mubayi and Wilson \cite{kmw}, Keller and the authors \cite{ekl}, and Keller and the second author \cite{keller-lifshitz}. These solutions have involved a wide range of methods (combinatorial, probabilistic, algebraic and Fourier-analytic), some of which have found important applications elsewhere. For example, the work of Frankl and F\"uredi \cite{ff} was one of the first uses of their widely-applicable `delta-system method', and the work of Keller and the second author \cite{keller-lifshitz} involved a broad extension of the `junta method', which we will discuss below.
 
 In this paper, we obtain forbidden intersection theorems for families of permutations. The study of Erd\H{o}s-Ko-Rado problems for families of permutations goes back to Deza and Frankl \cite{df}. For $n \in \mathbb{N}$, we let $S_n$ denote the symmetric group on $\{1,2,\ldots,n\} =:[n]$, i.e.\ the set of all bijections from $[n]$ to itself. We say a family of permutations $\f \subset S_n$ is {\em intersecting} if any two permutations in $\f$ agree on at least one point, i.e.\ if for any $\sigma,\tau \in \f$, there exists $i \in [n]$ such that $\sigma(i)=\tau(i)$. Deza and Frankl proved that if $\f \subset S_n$ is intersecting, then $|\f| \leq (n-1)!$. (Cameron and Ku \cite{ck} later proved that equality holds only if $\f$ consists of the coset of a stabilizer of a point.)
 
 We say a family of permutations $\f \subset S_n$ is {\em $t$-intersecting} if any two permutations in $\f$ agree on at least $t$ points, i.e.\ if for any $\sigma,\tau \in \f$, we have $|\{i \in [n]: \sigma(i)=\tau(i)\}| \geq t$. Deza and Frankl conjectured in \cite{df} that for any $t \in \mathbb{N}$, if $n$ is sufficiently large depending on $t$, then any $t$-intersecting family $\f \in S_n$ satisfies $|\f| \leq (n-t)!$ (which is the size of the $t$-intersecting family $\{\sigma \in S_n: \sigma(i)=i \ \forall i \in [t]\}$). This remained open for 31 years, and became known as the Deza-Frankl conjecture. It was proved by Friedgut and Pilpel, and independently and simultaneously by the first author, in 2008 (see \cite{efp}). The proof uses a spectral argument, combined with non-Abelian Fourier analysis on $S_n$ (a.k.a.\ representation theory of $S_n$); it is, to our knowledge, the first application of non-Abelian Fourier analysis to prove an exact result in Extremal Combinatorics. We remark that it was also claimed in \cite{efp} that equality holds in the Deza-Frankl conjecture only if $\f$ is the pointwise stabilizer of a $t$-element set (provided $n$ is large enough depending on $t$). This statement is true, but unfortunately, the proof of the equality case in \cite{efp} contained a hole (pointed out by Filmus \cite{filmus-comment}). An alternative proof of the equality case (and of a much stronger `stability' statement) is to be found in \cite{ellis-dfs}.

The Erd\H{o}s-S\'os forbidden intersection problem has a natural analogue for families of permutations. For $t \in \mathbb{N}$, we say a family of permutations $\f \subset S_n$ is {\em $(t-1)$-intersection-free} if no two permutations in $\f$ agree on exactly $t-1$ points, i.e.\ if for any two distinct $\sigma,\tau \in \f$, we have $|\{i \in [n]:\ \sigma(i)=\tau(i)\}| \neq t-1$. Clearly, a $t$-intersecting family is $(t-1)$-intersection-free, and the first author conjectured in \cite{ellis-forbidden} that for any $t \in \mathbb{N}$, the conclusion of the Deza-Frankl conjecture holds under the weaker condition that $\f$ is $(t-1)$-intersection-free (provided $n$ is sufficiently large depending on $t$). In the case $t=1$, a family is $t$-intersecting if and only if it is $(t-1)$-intersection-free, so there is nothing to prove in this case. The first author verified his conjecture in the case $t=2$, using a spectral argument to prove the conjecture asymptotically, and then a stability argument to prove it exactly (in this case). However, these techniques broke down for all $t >2$, and all other cases of the conjecture remained open.

It is in place to remark that recently, Keevash and Long \cite{kl} considered $t$-intersection-free families in $S_n$, where $t = \Theta(n)$. They proved that for any $\epsilon >0$, there exists $\delta = \delta(\epsilon)>0$ such that if $\epsilon n \leq t \leq (1-\epsilon)n$, and $\f \subset S_n$ is $t$-intersection-free, then $|\f| \leq (n!)^{1-\delta}$. (For each $\epsilon >0$, this result is sharp up to the value of $\delta = \delta(\epsilon)$, as is evidenced by applying Tur\'an's theorem to the Cayley graph on $S_n$ generated by the set of permutations with exactly $t$ fixed points. The dependence of $\delta(\epsilon)$ upon $\epsilon$, however, is likely to be far from best-possible.)

In this paper, we develop new, more robust techniques that enable us to solve the above-mentioned forbidden intersection conjecture of the first author, which strengthens the Deza-Frankl conjecture and resolves the permutation analogue of the Erd\H{o}s-S\'os forbidden intersection problem (for large $n$). In fact, we prove the following.
 
 \begin{thm}
\label{thm:extremal-forbidden}
If $t \in \mathbb{N}$, $n$ is sufficiently large depending on $t$, and $\f \subset S_n$ is $(t-1)$-intersection-free, then $|\f| \leq (n-t)!$, with equality only if $\f$ consists of a coset of the pointwise-stabilizer of a $t$-element set.
\end{thm}
 
 Our main tool for proving Theorem \ref{thm:extremal-forbidden} is the following `junta approximation' result. To state it, we need some more notation and terminology. If $S,T \subset \{1,2,\ldots,n\}$ with $|S|=|T|$, and $\pi: S \to T$ is a bijection, the \emph{$\pi$-star in $S_n$} is the family of all permutations in $S_n$ that agree with $\pi$ pointwise on all of $S$. An {\em $s$-star} is a $\pi$-star such that $\pi$ is a bijection between sets of size $s$. (Note that an $s$-star is precisely a coset of the pointwise-stabilizer of an $s$-element set; in this paper, we henceforth use the `star' terminology, as it is more concise.) If $n$ is understood, and for each $i \in [l]$, $S_i,T_i \subset [n]$ and $\pi_i:S_i \to T_i$ is a bijection, we define
$$\langle \pi_1,\ldots,\pi_l \rangle := \{\sigma \in S_n:\ (\exists i \in [l])(\forall j \in S_i)(\sigma(j)=\pi(j))\},$$
i.e., $\langle \pi_1,\ldots,\pi_l \rangle$ is the set of all permutations in $S_n$ that agree everywhere with at least one of the bijections $\pi_i$. We say that $\j \subset S_n$ is a {\em $C$-junta} if $\j = \langle \pi_1,\ldots,\pi_l\rangle$ for some bijections $\pi_i :S_i \to T_i$, where $l \leq C$ and $|S_i| \leq C$ for all $i \in [l]$. We remark that, if $\j \in S_n$ is a $C$-junta, then there exists a set $J \subset [n]$ with $|J| \leq C^2$ such that for any $\sigma \in S_n$, the question of whether or not $\sigma \in \j$ depends only upon the ordered set $(\sigma(j):\ j \in J)$; this may justify the use of the term `junta'. We think of $C$ as (an upper bound on) the `complexity' of the junta $\j$.

We can now state our `junta approximation' result.

\begin{thm}
\label{thm:junta-approximation}
For any $r,t \in \mathbb{N}$, there exists $C=C(r,t) \in \mathbb{N}$ such that if $\f \subset S_n$ is $(t-1)$-intersection-free, there exists a $t$-intersecting $C$-junta $\j \subset S_n$ such that $|\f \setminus \j| \leq Cn!/n^r$.
\end{thm}

Informally, this theorem says that any $(t-1)$-intersection-free family is `almost' contained within a $t$-intersecting junta of bounded complexity.

We remark that the above definition of a `junta' in $S_n$ is a natural analogue of a monotone increasing junta in $\mathcal{P}([n])$. Some more terminology: for a set $X$, we write $\mathcal{P}(X)$ for the power-set of $X$, and for $j \in \mathbb{N}$, we say a family of subsets $\mathcal{F} \subset \mathcal{P}([n])$ is a {\em $j$-junta} if there exists a set $J \subset [n]$ with $|J| \leq j$, and a family $\mathcal{G} \subset \mathcal{P}(J)$, such that for any $S \subset [n]$, we have $S \in \mathcal{F}$ if and only if $S \cap J \in \mathcal{G}$ (in other words, the membership in $\mathcal{F}$ of a set $S$ is determined purely by $S \cap J$). Thus, informally speaking, we may refer to $\mathcal{F}$ as a `junta' if it depends only upon a bounded number of coordinates. We say a family $\mathcal{F} \subset \mathcal{P}([n])$ is {\em monotone increasing} if whenever $S \in \f$ and $T \supset S$, we have $T \in \f$.

Juntas in $\mathcal{P}([n])$ have been widely used in Extremal Combinatorics and Theoretical Computer Science over the last twenty years; in particular, many theorems in Extremal Combinatorics have been proved by first approximating families with a given property by juntas, and then using a `perturbation' (or `stability') argument to show that the largest families with the given property are juntas. (This is known as the `junta method'.) One of the first uses of the junta method in Extremal Combinatorics was in the work of Dinur and Friedgut \cite{dinur-friedgut} on the approximation of intersecting families by `dictatorships', and other juntas. Recently, Keller and the second author \cite{keller-lifshitz} substantially generalised the junta method to deal with a large class of Tur\'an-type problems for hypergraphs (known as {\em Tur\'an problems for expansions}), and indeed, part of our strategy in this paper is to generalize some of the arguments in \cite{keller-lifshitz} to the symmetric group setting, demonstrating the flexibility of the aforesaid arguments.

Theorem \ref{thm:extremal-forbidden} follows from Theorem \ref{thm:junta-approximation} via a short and purely combinatorial argument. The main work of this paper is therefore in proving Theorem \ref{thm:junta-approximation}.

Our proof of Theorem \ref{thm:junta-approximation} employs a mixture of combinatorial, probabilistic and algebraic techniques. The first step is to prove a weak regularity lemma for families of permutations (Proposition \ref{prop:Regularity lemma}). This states that for {\em any} family of permutations $\f \subset S_n$ and any $r,s \in \mathbb{N}$, there exists an $O_{r,s}(1)$-junta 
$$\j = \langle \pi_1,\ldots,\pi_l\rangle$$
such that for each $i \in [l]$, the `slice'
$$\f(\pi_i) : = \{\sigma \in \f:\ \sigma(j)=\pi_i(j)\quad \forall j \in \Domain(\pi_i)\}$$
consisting of all the permutations in $\f$ agreeing pointwise with $\pi_i$, satisfies a weak pseudorandomness condition, which we term {\em $(s,n^{-r})$-uncaptureability}. (We say that $\mathcal{F}(\pi_i)$ is {\em $(s,n^{-r})$-uncaptureable} if for any bijection $\pi$ with domain disjoint from that of $\pi_i$, at least an $(n^{-r})$-fraction of $\f(\pi_i)$ consists of permutations disagreeing everywhere with $\pi$.) To prove Theorem \ref{thm:junta-approximation}, it then suffices to show that, provided $n$ is sufficiently large depending on $r$ and $t$, if $\f \subset S_n$ is $(t-1)$-intersection free, the junta $\j$ supplied by our weak regularity lemma (applied to $\f$ with $s=2r-1$) is $t$-intersecting. This is equivalent (for $n$ large enough) to the statement that for any $i,j \in [l]$, the bijections $\pi_i$ and $\pi_j$ have domains of size at least $t$ and agree with one another on at least $t$ points. 

We then assume, for the sake of a contradiction, that $\pi_1$ and $\pi_2$ agree with one another on less than $t$ points. The first step is a `bootstrapping' step: given our $(s,n^{-r})$-uncaptureable families $\f(\pi_1)$ and $\f(\pi_2)$, we find large subfamilies $\f_i' \subset \f(\pi_i)$ (for each $i \in \{1,2\}$) that satisfy a stronger (but still combinatorial) notion of psuedorandomness, which we term {\em $(r,\epsilon)$-quasirandomness}: this says, roughly, that restricting the family only to those permutations agreeing with a uniform random bijection with domain of size $r$, cannot change the measure of the family by very much, on average. The bootstrapping step is achieved via combinatorial techniques, similar to those used in \cite{keller-lifshitz} in the setting of uniform hypergraphs. To simplify the later (algebraic) part of the proof, at this point we make two reductions. Firstly, we use a probabilistic argument to reduce to the case where $t=1$: specifically, we use a random sampling method to show that, at the cost of passing to very slightly smaller subfamilies satisfying a very slightly weaker quasirandomness constraint, we can replace $\pi_1$ and $\pi_2$ by two bijections with any specified (but bounded) number of extra agreements between them. Secondly, we use a combinatorial argument to reduce to the case where $\pi_1$ and $\pi_2$ have the same domain and the same range. 

The next step is to show that our notion of combinatorial quasirandomness implies an algebraic notion of quasirandomness. This step involves some classical results from the representation theory of the symmetric group, combined with a Cauchy-Schwarz argument. (We believe that this step may find independent applications.)

The final step is to use our algebraic quasirandomness property, combined with a spectral argument, to complete the proof that $\j$ is $t$-intersecting. Specifically, this step involves showing that if $f,g:S_{n} \to [0,1]$ are sufficiently (algebraically) quasirandom and have sufficiently large expectations, then there must exist two permutations $\sigma_1,\sigma_2 \in S_n$ such that $\sigma_1$ and $\sigma_2$ disagree everywhere, and $f(\sigma_1)g(\sigma_2) >0$. (Because of our combinatorial reductions in previous steps, we have to work with $[0,1]$-valued functions on $S_n$, i.e.\ `fractional subfamiles' of $S_n$, in this step, rather than with straighforward subfamilies of $S_n$.) The representation-theoretic bounds needed in this step are much simpler than those employed in \cite{efp}, largely because of the extra quasirandomness condition we can use.

We also use our junta approximation theorem to obtain a new and rather strong `1\% stability' result for families of permutations with a forbidden intersection. Stability results in Extremal Combinatorics describe the structure of `large' families of objects with a given mathematical property. Suppose we are studying a set of mathematical objects $U$, in which each object has an `order' in $\mathbb{N}$ (with $U_n$ denoting the set of objects in $U$ of order $n$), and we have an extremal theorem specifying, for each $n \in \mathbb{N}$, the maximum possible size $M(n)$ of a subfamily $\f$ of $U_n$ such that $\f$ has the mathematical property $P$. A `99\% stability result', in this context, describes the structure of subfamilies of $U_n$ with property $P$ that have size at least $(1-\epsilon)M(n)$, for $\epsilon$ sufficiently small (usually, giving non-trivial information whenever $\epsilon$ at most some small, positive absolute constant). A `1\% stability result', on the other hand, describes the structure of subfamilies of $U_n$ with property $P$ that have size at least $\epsilon M(n)$, giving non-trivial structural information even when $\epsilon$ tends to zero with $n$ (provided $\epsilon$ tends to zero sufficiently slowly with $n$). In our case, $U_n=S_n$ for each $n \in \mathbb{N}$, the property $P$ in question is that of being $(t-1)$-intersection-free, and $M(n) = (n-t)!$ for all $n$ sufficiently large depending on $t$. In \cite{ellis-dfs}, the first author obtained a 99\% stability result for $t$-intersecting families of permutations. We obtain the following 1\% stability result for $(t-1)$-intersection-free families of permutations.

\begin{thm}
\label{thm:stability}
For any $r,t \in \mathbb{N}$ such that $r > t$, there exists $K = K(r,t)>0$ such that the following holds. If $\f \subset S_n$ is $(t-1)$-intersection free with $|\f| \geq K(n-t-1)!$, then there exists a $t$-star $\g \subset S_n$ such that $|\f \setminus \g| \leq K(n-r)!$.
\end{thm}

Even the weaker version of Theorem \ref{thm:stability}, where $(t-1)$-intersection-free is replaced by $t$-intersecting, was out of the reach of previous methods.

Given that our techniques in this paper are more robust than those in \cite{efp}, it is natural to ask whether they generalise to other infinite families of finite groups. This seems to be the case, and in a forthcoming paper with Guy Kindler, the authors prove analogues of Theorem \ref{thm:extremal-forbidden} for $\textrm{GL}(n,\mathbb{F}_q)$ and $\textrm{SL}(n,\mathbb{F}_q)$, for $n$ sufficiently large depending upon $q$ and $t$, using somewhat similar techniques to those in this paper, combined with a new hypercontractivity argument which may be of interest in its own right. The analogous problems for the orthogonal and symplectic groups over finite fields, on the other hand, currently seem out of reach.

The remainder of this paper is structured as follows. In Section \ref{sec:background}, we outline the background and tools we need from the representation theory of the symmetric group. In Section \ref{sec:reg-lemma}, we state and prove our weak regularity lemma for families of permutations. In Section \ref{sec:combinatorial-tools}, we develop the combinatorial and probabilistic tools we will need for the bootstrapping step described above. In Section \ref{sec:quasi}, we prove that combinatorial quasirandomness implies algebraic quasirandomness. In Section \ref{sec:proof}, we pull these ingredients together to prove Theorem \ref{thm:junta-approximation}. In Section \ref{sec:deduction}, we give the (fairly short) deduction of Theorem \ref{thm:extremal-forbidden} from Theorem \ref{thm:junta-approximation}, and we also deduce a new stability result from Theorem \ref{thm:junta-approximation}. Finally, in Section \ref{sec:conc}, we conclude with some open problems.

\section{Background and tools from the representation theory of the symmetric group}
\label{sec:background}

Our treatment of the representation theory of $S_n$ follows James and Kerber \cite{james-kerber}; the reader is referred to this book for more background.

We equip $\mathbb{R}[S_n]$ with the inner product $\langle\ ,\ \rangle$ induced by the uniform measure on $S_n$: for $f,g:S_n\to \mathbb{R}$,
$$\langle f,g \rangle := \frac{1}{n!}\sum_{\sigma \in S_n} f(\sigma) g(\sigma).$$
We let $\|\cdot \|_2$ denote the corresponding Euclidean norm.

If $n \in \mathbb{N}$, an (integer) {\em partition} $\alpha$ of $n$ is a non-increasing sequence of positive integers with sum $n$, i.e.\ $\alpha = (\alpha_1,\ldots,\alpha_l)$ for some $l \in \mathbb{N}$, with $\alpha_i \in \mathbb{N}$ for all $i \in [l]$, $\alpha_1 \geq \ldots \geq \alpha_l$, and $\sum_{i=1}^{l} \alpha_i=n$. If $\alpha$ is a partition of $n$, we write $\alpha \vdash n$. We write $>$ for the lexicographic ordering on partitions, i.e.\ $\alpha > \beta$ if $\alpha_j > \beta_j$ where $j = \min\{i:\ \alpha_i \neq \beta_i\}$.

Recall that the equivalence classes of irreducible complex representations of $S_n$ are in a natural one-to-one correspondence with the partitions of $n$, with each partition $\alpha \vdash n$ corresponding to the (irreducible) {\em Specht module} $S^{\alpha}$. For each $\alpha \vdash n$, let $\chi_{\alpha}$ denote the character of $S^{\alpha}$; since $S^{\alpha}$ is also a real representation, $\chi_{\alpha}$ is real-valued. Let $\mathsf{f}^{\alpha}: = \chi_{\alpha}(\Id)$ denote the dimension of $S^{\alpha}$, and let $W_{\alpha}$ denote the sum of all copies of $S^{\alpha}$ in $\mathbb{R}^{S_n}$; then $\dim[W_{\alpha}] = (\mathsf{f}^{\alpha})^2$. We have the orthogonal decomposition
$$\mathbb{R}^{S_n} = \bigoplus_{\alpha \vdash n} W_{\alpha},$$
and the orthogonal projection $P_{\alpha}(f)$ of a function $f:S_n\to \mathbb{R}$ onto $W_{\alpha}$ is given by the formula
$$P_{\alpha}(f)(\sigma) = \frac{\mathsf{f}^{\alpha}}{n!} \sum_{\pi \in S_n} f(\pi) \chi_{\alpha}(\sigma \pi^{-1}) \quad \forall \sigma \in S_n.$$
Hence, for any $\alpha \vdash n$ and any $f:S_n\to \mathbb{R}$, we have
\begin{equation}\label{eq:inner-prod}\|P_{\alpha}(f)\|_2^2= \langle P_{\alpha}(f),f \rangle = \frac{\mathsf{f}^{\alpha}}{(n!)^2} \sum_{\sigma,\pi \in S_n} f(\sigma)f(\pi) \chi_{\alpha}(\sigma \pi^{-1}).\end{equation}

\begin{defn} If \(\lambda = (\lambda_1,\ldots,\lambda_l)\) is a partition of \(n\), the {\em Young diagram of shape \(\lambda\)} is the array of $n$ left-justified cells with \(\lambda_i\) cells in row \(i\), for each \(i \in [l]\).
\end{defn}
For example, the Young diagram of the partition \((3,2^{2})\) is:
  \[
\yng(3,2,2)
\]

\begin{defn}
If $\lambda$ is a partition of $n$, its {\em transpose} $\lambda'$ is the partition of $n$ whose Young diagram is the transpose of the Young diagram of $\lambda$, i.e.\ its Young diagram is obtained by interchanging the rows and columns of the Young diagram of $\lambda$.
\end{defn}

\begin{defn}
A {\em \(\lambda\)-tableau} is a Young diagram of shape \(\lambda\), each of whose cells contains a number between 1 and \(n\). If \(\mu = (\mu_1,\ldots,\mu_l)\) is a partition of \(n\), a Young tableau is said to have {\em content \(\mu\)} if it contains \(\mu_i\) \(i\)'s for each \(i \in [l]\).
\end{defn}

\begin{defn}
A Young tableau is said to be {\em standard} if it has content \((1,1,\ldots,1)\) and the numbers are strictly increasing down each row and along each column.
\end{defn}
\begin{defn}
A Young tableau is said to be {\em semistandard} if the numbers are non-decreasing along each row and strictly increasing down each column.
\end{defn}
The relevance of standard Young tableaux stems from the following.
\begin{thm}
\label{thm:dimension}
If \(\lambda\) is a partition of \(n\), then \(\mathsf{f}^\lambda\) is the number of standard \(\lambda\)-tableaux.
\end{thm}
It follows from Theorem \ref{thm:dimension} that $\mathsf{f}^{\lambda'} = \mathsf{f}^{\lambda}$ for all $\lambda \vdash n$.

\begin{defn}
  The {\em hook} of a node $(i,j)$ in the Young diagram of a partition $\lambda$
  is $H_{i,j} = \{ (i,j') : j' \geq j \} \cup \{ (i',j) : i' \geq i \}$.
  The {\em hook length} of $(i,j)$ is $h_{i,j} = |H_{i,j}|$.
\end{defn}

\begin{thm} \label{hook-formula}
  (The Hook Formula.) If \(\lambda \vdash n\) with hook lengths $h_1, \ldots, h_k$, then
  \begin{equation}
  \label{eq:hook}
    \mathsf{f}^{\lambda} = \frac{n!}{\prod_i h_i}.
  \end{equation}
\end{thm}

We say that two \(\lambda\)-tableaux of content $(1,1,\ldots,1)$ are {\em row-equivalent} if they contain the same set of numbers in each row. A row-equivalence-class of \(\lambda\)-tableaux with content $(1,1,\ldots,1)$ is called a {\em \(\lambda\)-tabloid}. If $\lambda=(\lambda_1,\ldots,\lambda_l)$, then a $\lambda$-tabloid is simply an ordered $l$-tuple $(S_1,\ldots,S_l)$ of sets where $S_i \in [n]^{(\lambda_i)}$ for all $i \in [l]$ and the sets $S_1,\ldots,S_l$ partition $[n]$. Let us write $s(\lambda)$ for the number of permutations in $S_n$ that stabilize a fixed $\lambda$-tabloid; then
$$s(\lambda) = \prod_{i=1}^{l} \lambda_i!.$$

Write $\t_{\lambda}$ for the set of all $\lambda$-tabloids. Consider the natural left action of \(S_n\) on the set $\t_{\lambda}$ of all \(\lambda\)-tabloids, i.e.\ if a $\lambda$-tabloid $T$ has $i$th row $R_i \in [n]^{(\lambda_i)}$ for each $i$, then $\sigma(T)$ has $i$th row $\sigma(R_i)$ for each $i$. Let \(M^{\lambda}\) denote the induced permutation representation. We write \(\xi_{\lambda}\) for the character of \(M^{\lambda}\); the \(\xi_{\lambda}\) are called the {\em permutation characters} of \(S_n\). Note that if $\sigma \in S_n$, then $\xi_{\lambda}(\sigma)$ is simply the number of $\lambda$-tabloids fixed by $\sigma$.

{\em Young's theorem} gives the decomposition of each permutation representation into irreducible representations of \(S_n\), in terms of the {\em Kostka numbers}.

\begin{defn}
Let \(\lambda\) and \(\mu\) be partitions of \(n\). The {\em Kostka number} \(K_{\lambda,\mu}\) is the number of semistandard \(\lambda\)-tableaux of content \(\mu\).
\end{defn}

\begin{thm}[Young's theorem]
\label{thm:young}
If \(\mu\) is a partition of \(n\), then
\[M^{\mu} \cong \bigoplus_{\substack{\lambda \vdash n\colon\\ \lambda \geq \mu\hphantom{\colon}}} K_{\lambda, \mu}
  S^{\lambda}.\]
\end{thm}

It follows that for each partition \(\mu\) of \(n\), we have
\[\xi_{\mu} = \sum_{\substack{\lambda \vdash n\colon\\ \lambda \geq \mu\hphantom{\colon}}} K_{\lambda,\mu} \chi_{\lambda}.\]

On the other hand, we can express the irreducible characters in terms of the permutation characters using the {\em determinantal formula}: for any partition \(\alpha\) of \(n\),
\begin{equation}\label{eq:determinantalformula} \chi_{\alpha} = \sum_{\pi \in S_{n}} \sgn(\pi) \xi_{\alpha - \textrm{id}+\pi}.\end{equation}
Here, if \(\alpha = (\alpha_{1},\alpha_{2},\ldots,\alpha_{l})\), \(\alpha - \textrm{id}+\pi\) is defined to be the sequence
\[(\alpha_{1}-1+\pi(1),\alpha_{2}-2+\pi(2),\ldots,\alpha_{n}-n+\pi(n)),\]
where $\alpha_i:=0$ for all $i > l$. If this sequence has all its entries non-negative, we let \(\overline{\alpha-\textrm{id}+\pi}\) be the partition of \(n\) obtained by reordering its entries into non-increasing order and deleting the zero entries, and we define \(\xi_{\alpha - \textrm{id}+\pi} = \xi_{\overline{\alpha-\textrm{id}+\pi}}\). If the sequence has a negative entry, we define \(\xi_{\alpha - \textrm{id}+\pi} = 0\). Note that if \(\xi_{\beta}\) appears on the right-hand side of (\ref{eq:determinantalformula}), then \(\beta \geq \alpha\), so the determinantal formula expresses \(\chi_{\alpha}\) in terms of \(\{\xi_{\beta}: \ \beta \geq \alpha\}\). Observe that if $\alpha_1 \geq n-r$, then $\xi_{\alpha - \textrm{id}+\pi} \neq 0$ only if $\pi$ fixes $[n] \setminus [r+1]$ pointwise, so the sum on the right-hand side of (\ref{eq:determinantalformula}) has at most $(r+1)!$ non-zero terms.

If $G=(V,E)$ is a finite graph, its {\em adjacency matrix} is the 0-1 matrix with rows and columns indexed by $V$, and with $(u,v)$-entry equal to 1 if and only if $uv \in E(G)$, for each $u,v \in V(G)$. If $G$ is a $d$-regular graph, its {\em normalized adjacency matrix} is the matrix obtained by dividing each entry of the adjacency matrix by $d$, so that all row and column sums are equal to 1.

If $\Gamma$ is a finite group, and $S \subset \Gamma$ with $S^{-1}=S$ and $\Id \notin S$, we write $\Cay(\Gamma,S)$ for the (right) {\em Cayley graph of $\Gamma$ with respect to $S$}, which is defined to be the graph with vertex-set $\Gamma$ and edge-set $\{\{g,gs\}:\ g \in \Gamma,\ s\in S\}$. If $S$ is invariant under conjugation by elements of $\Gamma$, then we say that $\Cay(\Gamma,S)$ is a {\em normal} Cayley graph.

The following observation is well-known, going back essentially to Frobenius and Schur (see for example \cite{ds}).
\begin{lem}
Let $\Gamma$ be a finite group, and let $S \subset \Gamma$ be conjugation-invariant with $S^{-1}=S$ and $\Id \notin S$. Let  $\mathcal{R}$ be a complete set of inequivalent complex irreducible representations of $\Gamma$. Then the eigenvalues of the adjacency matrix of the normal Cayley graph $\Cay(\Gamma,S)$ are given by
\begin{equation}\label{eq:frob} \lambda_{\rho} = \frac{1}{\dim(\rho)} \sum_{s \in S} \chi_{\rho}(s)\quad (\rho \in \mathcal{R}),\end{equation}
where $\chi_{\rho}$ is the character of the representation $\rho$.
\end{lem}

For $n \in \mathbb{N}$, let $\mathcal{D}_n$ denote the set of derangements of $[n]$ (the permutations in $S_n$ without fixed points), and let $d_n = |\mathcal{D}_n|$. It is a well-known and elementary consequence of the inclusion-exclusion formula that
$$d_n = \sum_{j=0}^{n} (-1)^j \frac{n!}{j!} =( 1/e+o(1))n!.$$
The Cayley graph $\Cay(S_n,\mathcal{D}_n)$ is called the {\em derangement graph} on $S_n$; it is a normal Cayley graph, so (\ref{eq:frob}) applies, and therefore the eigenvalues of its adjacency matrix are given by
\begin{equation} \label{eq:frob-der} \lambda_{\alpha} = \frac{1}{\mathsf{f}^{\alpha}} \sum_{\sigma \in \d_n} \chi_{\alpha}(\sigma)\quad (\alpha \vdash n).\end{equation}
We will use this fact in the next two lemmas.
\begin{lem}
\label{lem:eval-estimate}
Let $A$ denote the normalized adjacency matrix of $\Cay(S_n,\mathcal{D}_n)$ (normalized so that all row and column sums are equal to 1). For each $\alpha \vdash n$, let $\lambda_{\alpha}$ denote the corresponding eigenvalue of $A$. Let $\mathsf{f}^{\alpha} = \dim(S^{\alpha})$ for each $\alpha \vdash n$. Then
$$|\lambda_{\alpha}| = O(1/\mathsf{f}^{\alpha})\quad \forall \alpha \vdash n.$$
\end{lem}
\begin{proof}
The normalized version of (\ref{eq:frob-der}) is
$$\lambda_{\alpha} = \frac{1}{\mathsf{f}^{\alpha} d_n} \sum_{\sigma \in \d_n} \chi_{\alpha}(\sigma) = \frac{n!}{\mathsf{f}^{\alpha} d_n} \langle \chi_{\alpha},1_{\d_n}\rangle.$$
Applying the Cauchy-Schwarz inequality, and using the orthonormality of irreducible characters, we have
$$|\lambda_{\alpha}| \leq \frac{n!}{\mathsf{f}^\alpha d_n}\|\chi_{\alpha}\|_2 \|1_{\d_n}\|_2 = \frac{n!}{\mathsf{f}^{\alpha} d_n} \cdot 1 \cdot \sqrt{d_n/n!} = \frac{1}{\mathsf{f}^{\alpha}} \sqrt{n!/d_n}.$$
Since $d_n = (1/e+o(1))n!$, the lemma follows.
\end{proof}

\begin{remark}
One can also prove Lemma \ref{lem:eval-estimate} by using the fact that the sum of the squares of the eigenvalues of $A$ (repeated with their geometric multiplicities) is equal to the sum of the squares of the entries of $A$, which in turn is equal to $n!/d_n$.
\end{remark}

For $s,n \in \mathbb{N}$, let $\l_n(s)$ denote the set of all partitions of $n$ with largest part of size $s$. (When $n$ is understood, we will sometimes omit it, writing $\l_n(s)=\l(s)$. Similarly, let $\l_n({\geq}s)$ denote the set of partitions of $n$ with largest part of size at least $s$, and let $\l_n({<}s)$ denote the set of all partitions of $n$ with largest part of size less than $s$. Finally, let $\l_n^*({\geq}s)$ denote the set of all partitions of $n$ with largest part of size between $s$ and $n-1$, i.e.\ the set of all non-trivial partitions of $n$ with largest part of size at least $s$.

\begin{lem}
\label{lem:transpose}
Let $A$ denote the normalized adjacency matrix of $\Cay(S_n,\mathcal{D}_n)$, and for each $\alpha \vdash n$, let $\lambda_{\alpha}$ denote the corresponding eigenvalue of $A$. Let $r \in \mathbb{N}$. For each $\alpha \vdash n$ such that $\alpha' \in \l({\geq}n-r)$, we have
$$|\lambda_{\alpha}| = \frac{O_r(1)}{(n-1)!}.$$
\end{lem}
\begin{proof}
Let $\alpha \vdash n$ such that $\alpha' \in \l({\geq}n-r)$. Recall that
$$S^{\alpha} \cong S^{\alpha'} \otimes \sgn,$$
where $\sgn$ denotes the sign representation. Hence,
$$\chi_{\alpha}(\sigma) = \chi_{\alpha'}(\sigma) \sgn(\sigma) \quad \forall \sigma \in S_n.$$
Recalling that $\mathsf{f}^{\alpha'} = \mathsf{f}^{\alpha}$, we have
$$\lambda_{\alpha} =\frac{1}{\mathsf{f}^{\alpha} d_n} \sum_{\sigma \in \d_n} \chi_{\alpha'}(\sigma) \sgn(\sigma) = \frac{1}{\mathsf{f}^{\alpha'} d_n} \sum_{\sigma \in \d_n} \chi_{\alpha'}(\sigma) \sgn(\sigma).$$
By reducing $r$ if necessary, we may assume that $\alpha' \in \l(n-r)$, i.e.\ that $(\alpha')_1=n-r$, so that $\mathsf{f}^{\alpha'} = \Theta_r(n^r)$, by the Hook Formula (\ref{eq:hook}). Let
$$\chi_{\alpha'} = \sum_{\beta \geq \alpha'} c_{\beta}\xi_{\beta}$$
be the expression of $\chi_{\alpha'}$ as a linear combination of the permutation characters, as guaranteed by the determintantal formula (\ref{eq:determinantalformula}). As remarked above, in this case, the sum in the determinantal formula has at most $(r+1)!$ non-zero summands, and therefore
$$\sum_{\beta \geq \alpha'} |c_{\beta}| \leq (r+1)! = O_r(1).$$
It therefore suffices to show that for each $\beta \in \l({\geq}n-r)$, we have
$$\frac{1}{\mathsf{f}^{\alpha'} d_n} \sum_{\sigma \in \d_n} \xi_{\beta}(\sigma) \sgn(\sigma) = O_r(1/(n-1)!).$$
Writing $\beta = (n-s,\beta_2,\ldots,\beta_l)$, where $s \leq r$, we have
\begin{align*} \sum_{\sigma \in \d_n} \xi_{\beta}(\sigma) \sgn(\sigma) & = \sum_{T \in \t_{\beta}} \sum_{\sigma \in \d_n} 1_{\{\sigma(T)=T\}} \sgn(\sigma)\\
& = |\t_{\beta}| \sum_{\sigma_2 \in \d_{\beta_2},\ldots, \sigma_l \in \d_{\beta_l}} \sgn(\sigma_2)\cdot \sgn(\sigma_3) \ldots \cdot \sgn(\sigma_l) \sum_{\sigma_1 \in \d_{n-s}} \sgn(\sigma_1)\\
& = O_r(1) (-1)^{n-s-1}(n-s-1) |\t_{\beta}|\\
& = O_r(n) |\t_{\beta}|,
\end{align*}
using the well-known fact that  $\sum_{\sigma \in \d_{m}} \sgn(\sigma) = (-1)^{m-1}(m-1)$ for all $m \in \mathbb{N}$. Hence,
\begin{align*}
\frac{1}{\mathsf{f}^{\alpha'} d_n} \sum_{\sigma \in \d_n} \xi_{\beta}(\sigma) \sgn(\sigma) & = \frac{|\t_{\beta}|}{\mathsf{f}^{\alpha'} d_n} O_r(n)\\
& = O_r(1/(n-1)!)
\end{align*}
as required, using the facts that $d_n = (1/e+o(1))n!$, $|\t_{\beta}| = O_s(n^s)$ and $\mathsf{f}^{\alpha'} = \Theta_r(n^{r})$.
\end{proof}

\begin{cor}
\label{cor:low-evals}
Let $A$ denote the normalized adjacency matrix of $\Cay(S_n,\mathcal{D}_n)$, and for each $\alpha \vdash n$, let $\lambda_{\alpha}$ denote the corresponding eigenvalue of $A$. For each $r \in \mathbb{N}$, we have
$$\max_{\alpha \in \mathcal{L}({<}n-r)} |\lambda_{\alpha}| = O_r(n^{-r-1}).$$
\end{cor}
\begin{proof}
It is well-known, and straightforward to prove (using the Hook Formula and induction, see e.g.\ Lemma 20 in \cite{setwise}) that
$$\min\{\mathsf{f}^{\alpha}:\ \alpha \vdash n,\ \alpha,\alpha' \in \l({<} n-r)\} = \Omega_r(n^{r+1}).$$
Hence, using Lemma \ref{lem:eval-estimate}, we have
$$\max\{|\lambda_{\alpha}|:\ \alpha \vdash n,\ \alpha,\alpha' \in \l({<} n-r)\} = O_r(n^{-r-1}).$$
On the other hand, Lemma \ref{lem:transpose} implies that
$$\max\{|\lambda_{\alpha}|:\ \alpha \vdash n,\ \alpha' \in \l({\geq} n-r)\} = O_r(1/(n-1)!).$$
Combining these two maxima yields the corollary.
\end{proof}

\section{A weak regularity lemma for families of permutations}
\label{sec:reg-lemma} 

In this section, we state and prove our weak regularity lemma for families of permutations. We first define the (very weak) notion of psuedorandomness that appears in this regularity lemma, viz, $(s,\epsilon)$-uncaptureability, as discussed in the Introduction.
  
First for some new notation. If $S_{1},S_{2},T_1,T_{2}\subset [n]$, and $\pi_{1}\colon S_{1}\to T_{1},\ \pi_{2}:S_{2}\to T_{2}$
are bijections, we write $S_n\left(\pi_{1},\overline{\pi_{2}}\right)$ for the set of all permutations in $S_n$ that agree with $\pi_{1}$ on every point of $S_{1}$
and disagree with $\pi_{2}$ on every point of $S_{2}$. Similarly, if $\pi_i:S_i \to T_i$ are bijections for each $i \in [l]$ and $\sigma_i:S_i'\to T_i'$ are bijections for each $i \in [m]$, we write $S_n(\pi_{1},\ldots,\pi_l,\overline{\sigma_1},\ldots,\overline{\sigma_{m}})$ for the set of all permutations in $S_n$ that agree with $\pi_{i}$ on every point of $S_{i}$ for each $i \in [l]$, and disagree with $\sigma_i$ on every point of $S_i'$, for each $i \in [m]$.

If $\f \subset S_n$, we write $\f\left(\pi_{1},\overline{\pi_{2}}\right)$
for the set of all permutations in $\f$ that agree with $\pi_{1}$ on every point of $S_{1}$
and disagree with $\pi_{2}$ on every point of $S_{2}$. (We define $\f(\pi_{1},\ldots,\pi_l,\overline{\sigma_1},\ldots,\overline{\sigma_{m}})$ similarly.) We regard the family $\f(\pi_1,\overline{\pi_2})$ as a subset of $S_n(\pi_1,\overline{\pi_2})$ and we equip the latter with the uniform measure, so that
\[
\mu\left(\f\left(\pi_{1},\overline{\pi_{2}}\right)\right):=\frac{\left|\f\left(\pi_{1},\overline{\pi_{2}}\right)\right|}{\left|S_{n}\left(\pi_{1},\overline{\pi_{2}}\right)\right|},
\]
provided $S_{n}\left(\pi_{1},\overline{\pi_{2}}\right) \neq \emptyset$. Similarly,
$$\mu(\f(\pi_{1},\ldots,\pi_l,\overline{\sigma_1},\ldots,\overline{\sigma_{m}})) := \frac{|\f(\pi_{1},\ldots,\pi_l,\overline{\sigma_1},\ldots,\overline{\sigma_{m}})|}{|S_n(\pi_{1},\ldots,\pi_l,\overline{\sigma_1},\ldots,\overline{\sigma_{m}})|}.$$

If $f:S_n \to \mathbb{R}$, $S,T \subset [n]$ and $\pi:S \to T$ is a bijection, we write $f(\pi)$ for the restriction of $f$ to $S_n(\pi)$, and we equip the latter with the uniform measure, so that
$$\mathbb{E}[f(\pi)]: = \frac{1}{|S_n(\pi)|} \sum_{\sigma \in S_n(\pi)}f(\sigma).$$

\begin{defn}
A family $ \f\subset S_{n}$ of permutations is said to be \emph{$\left(s,\epsilon\right)$-captureable}
if there exist sets $S,T \subset [n]$ with $|S|=|T| \leq s$, and a bijection $\pi\colon S\to T$,
such that $\mu\left(\f\left(\overline{\pi}\right)\right)\le\epsilon$. Similarly, if $\pi_1:S_1\to T_1,\ \pi_2:S_2\to T_2$ are bijections with $S_1,S_2,T_1,T_2 \subset [n]$, we say that a family $\f\left(\pi_{1},\overline{\pi_{2}}\right) \subset S_n(\pi_1,\overline{\pi_2})$
is {\em $\left(s,\epsilon\right)$-captureable} if there exist sets $S \subset [n] \setminus \Domain(\pi_1)$ and $T \subset [n] \setminus \Range(\pi_1)$, with $|S|=|T| \leq s$, and a bijection $\pi\colon S\to T$,
such that $\mu\left(\f\left(\pi_{1},\overline{\pi_{2}},\overline{\pi}\right)\right)\le\epsilon.$
If a family is not $\left(s,\epsilon\right)$-captureable, then we say it is $\left(s,\epsilon\right)$-\emph{uncaptureable}. 
\end{defn}

It is easy to check, using a Chernoff bound, that a $p(n)$-random subfamily of $S_n$ (in which every permutation in $S_n$ is included independently, with probability $p(n)$) is $(n,1/3)$-uncaptureable with high probability, provided $p(n) (n-1)!/\log n \to \infty$ as $n \to \infty$. In a sense, therefore, we are justified in viewing uncaptureability as a notion of pseudorandomness. However, it is a very weak notion of pseudorandomness (and indeed, we will need to `bootstrap' it to a much stronger notion): the 2-junta $\{\sigma \in S_n:\ \sigma(1) \in \{1,2\}\}$ is $(n/2,1/5)$-uncaptureable if $n$ is sufficiently large, for example, despite being very far from random-like. (The family $\{\sigma \in S_n:\ \sigma \text{ has at least one fixed point in }[s]\}$, on the other hand, is $(s,0)$-captureable.)

Here, then, is our weak regularity lemma.
\begin{prop}[A weak regularity lemma for sets of permutations]
\label{prop:Regularity lemma} For each $r,s \in \mathbb{N} \cup \{0\}$, there exists $C = C(r,s) \in \mathbb{N}$ such that for any family $\f\subset S_{n}$, there exists a $C$-junta $\j = \langle \pi_1,\ldots,\pi_l\rangle \subset S_n$ 
such that
\begin{enumerate}
\item $|\Domain(\pi_i)| <r$ for each $i \in [l]$;
\item $\mu(\f\setminus \j)\le Cn^{-r}$;
\item for each $i \in [l]$, the family $\f\left(\pi_{i}\right)$ is $(s,n^{-r})$-uncaptureable.
\end{enumerate}
\end{prop}
\begin{proof}
We construct a set $J$ of bijections such that the statement of the lemma holds with $\j = \langle J \rangle$. We construct $J$ iteratively, along with a labelled, rooted tree $\t$. Start with $J = \emptyset$, and with $\t$ consisting of a single node (the root), labelled $v_{\emptyset}$. If $\f$ itself is $(s,n^{-r})$-uncaptureable, then stop, declare $v_{\emptyset}$ to be a good leaf, and take $\j = S_n$. Otherwise, $\f$ is $(s,n^{-r})$-captureable, so there exist sets $S,T \in \binom{[n]}{ \leq s}$ and a bijection $\pi: S\to T$ such that $\mu\left(\f\left(\overline{\pi}\right)\right)\le n^{-r}$. For each $i \in S$, add a new node to $\t$ which is a child of $v_{\emptyset}$, labelled $v_{\sigma_i}$, where $\sigma_i$ is the bijection from $\{i\}$ to $\{\pi(i)\}$ mapping $i$ to $\pi(i)$. 

Now at any stage, if $\t$ has at least one leaf that has not yet been declared good or bad, choose one such. Suppose it is labelled $v_{\sigma}$ for some bijection $\sigma:U \to W$. If $|U|=r$, then declare $v_{\sigma}$ to be a bad leaf of $\t$. If $|U| < r$ and $\f(\sigma)$ is $(s,n^{-r})$-uncapturable, then add $\sigma$ to $J$ and declare $v_{\sigma}$ to be a good leaf of $\t$. If $|U| < r$ and $\f(\sigma)$ is $(s,n^{-r})$-capturable, then there exist sets $S \in \binom{[n] \setminus \Domain(\sigma)}{ \leq s}$ and $T \in \binom{[n] \setminus \Range(\sigma)}{ \leq s}$, and a bijection $\pi = \pi_{\sigma}: S\to T$ such that $\mu\left(\f\left(\sigma,\overline{\pi}\right)\right)\le n^{-r}$. For each $i \in S$, add a new node to $\t$ which is a child of $v_{\sigma}$, labelled $v_{\sigma_i}$, where $\sigma_i$ is the bijection from $U \cup \{i\}$ to $W \cup \{\pi(i)\}$ which agrees with $\sigma$ on $U$ and maps $i$ to $\pi(i)$. 

This process terminates when all leaves of $\t$ have been declared good or bad. At this stage, let $\j = \langle J \rangle$. (Note that $J$ consists of all the bijections labelling good leaves.) By the definition of `good', $\f(\pi)$ is $(s,n^{-r})$-uncaptureable for every $\pi \in J$. Note that every leaf of $\t$ has depth at most $r$ (relative to the root), since the depth of a leaf $v_{\sigma}$ is simply the cardinality of the domain of the corresponding bijection $\sigma$. Note also that every node has at most $s$ children. Hence, the tree $\t$ has at most $s^r$ leaves, so it has at most $s^r$ good leaves, and therefore $|J| \leq s^r$. Observe that for any permutation $\tau \in \f \setminus \j$, either $\tau$ agrees with $\sigma$ for some bad leaf $v_{\sigma}$, or else $\tau$ disagrees (in at least one place) with every bijection labelling a leaf. In the former case, $\tau \in \langle \sigma \rangle$, and $\mu(\langle \sigma \rangle) = (n-r)!/n! = O_r(n^{-r})$. In the latter case, let $v_{\sigma'}$ be an internal node of minimal depth such that $\sigma'$ agrees everywhere with $\tau$. Then $\tau \in \f(\sigma',\overline{\pi})$, where $\pi = \pi_{\sigma'}$ is chosen as above, so that $\mu(\f(\sigma',\overline{\pi})) \leq n^{-r}$. Since there are at most $s^r$ possibilities for bad leaves $v_{\sigma}$ and at most $(s^r-1)/(s-1) \leq s^r$ possibilities for internal nodes $v_{\sigma'}$, the union bound implies that $\mu(\f \setminus \j) = O_{r,s}(n^{-r})$, as required.
\end{proof}

\section{Combinatorial and probabilistic tools for finding `highly regular' subfamilies of uncaptureable families}
\label{sec:combinatorial-tools}

In this section, we describe the combinatorial and probabilistic tools we will need to `bootstrap' our weak notion of psuedorandomness (uncaptureability) into a stronger notion of pseudrandomness.

First, we need some more notation. If $\pi_1:S_1\to T_1,\ \pi_2:S_2\to T_2$ are bijections, with $S_1,S_2,T_1,T_2 \subset [n]$, we write $\pi_1 \cap \pi_2$ for the restriction of $\pi_1$ (or equivalently, of $\pi_2$) to the set of elements of $S_1 \cap S_2$ at which $\pi_1$ and $\pi_2$ agree. Note that, if one uses the formal definition of a function as a set of ordered pairs, $\pi_1 \cap \pi_2$ is simply the set-theoretic intersection of the sets $\pi_1$ and $\pi_2$ of ordered pairs. Hence, it is no abuse of notation to write $|\pi_1 \cap \pi_2|$ for the size of the domain of $\pi_1 \cap \pi_2$, i.e.\ the number of elements of $S_1 \cap S_2$ at which $\pi_1$ and $\pi_2$ agree, and we use this useful convention in the sequel. If $\pi_1$ and $\pi_2$ do not conflict anywhere, i.e.\ there exists no $x \in S_1 \cap S_2$ such that $\pi_1(x) \neq \pi_2(x)$ and no $y \in T_1 \cap T_2$ such that $\pi_1^{-1}(y) \neq \pi_2^{-1}(y)$, then we write $\pi_1 \cup \pi_2$ for the bijection from $S_1 \cup S_2$ to $T_1 \cup T_2$ that agrees with $\pi_1$ on $S_1$ and with $\pi_2$ on $S_2$. (Note, again, that if one uses the formal definition of a function as a set of ordered pairs, $\pi_1 \cup \pi_2$ is simply the set-theoretic union of the sets $\pi_1$ and $\pi_2$, under the no-conflict hypothesis.)

We are now ready to define a combinatorial notion of pseudorandomness (which we term `quasiregularity') that is significantly stronger than uncaptureability.

\begin{defn}
Let $s \leq n$, let $\alpha \geq 1$ and let $\pi_1:S_1\to T_1,\ \pi_2:S_2\to T_2$ be bijections with $S_1,S_2,T_1,T_2 \subset [n]$. A family $\f\left(\pi_{1},\overline{\pi_{2}}\right) \subset S_{n}$
is {\em $\left(s,\alpha\right)$-quasiregular} if there exists no bijection $\pi:S \to T$
such that $S \subset [n] \setminus \Domain(\pi_1)$, $T \subset [n] \setminus \Range(\pi_1)$, $|S|=|T| =s$, and $\mu\left(\f\left(\pi_{1},\overline{\pi_{2}},\pi\right)\right)\ge\alpha\mu\left(\f\right)$. Similarly, a function $f:S_n(\pi_{1},\overline{\pi_{2}}) \to \mathbb{R}$ is {\em $\left(s,\alpha\right)$-quasiregular} if there exists no bijection $\pi:S \to T$ such that $S \subset [n] \setminus \Domain(\pi_1)$, $T \subset [n] \setminus \Range(\pi_1)$, $|S|=|T| = s$, and $\mathbb{E}[f(\pi)] \geq \alpha \mathbb{E}[f]$.
\end{defn}

Informally, a family of permutations in $S_n$ (or a function on $S_n$) is highly quasiregular if, by restricting the family (or, respectively, the function) to the set of all permutations that agree with a bijection with domain and range of bounded size, one cannot increase the measure of the family (or, respectively, the expectation of the function) by too much. It is easy to check, using a Chernoff bound, that a $p(n)$-random subfamily of $S_n$ is $(n-k(n),1+\epsilon(n))$-quasiregular with high probability, provided $p(n) \epsilon(n)^2 k(n)! /(n \log n) \to \infty$ as $n \to \infty$. On the other hand, any non-trivial $C$-junta in $S_n$ is $(1,n/C)$-quasiregular. The notion of quasiregularity therefore captures much more of what we mean by a `random' family, than uncaptureability does!

We will need the following two straightforward claims about quasiregular families. These say that if one restricts a quasiregular family to the set of permutations disagreeing everywhere with a bijection with small domain, the quasiregularity and measure of the family is not much affected.

\begin{claim}
\label{claim:quasi}
Let $b,n,r \in \mathbb{N}$, and let $\alpha \geq 1$ such that $\alpha \leq (n-b)/(4r)$. Let $\sigma$ and $\rho$ be bijections between subsets of $[n]$, with $|\Domain(\sigma)|+|\Domain(\rho)|\leq b$. Suppose that $\h(\sigma,\overline{\rho}) \subset S_n(\sigma,\overline{\rho})$ is $(1,\alpha)$-quasiregular, and let $\pi$ be a bijection between subsets of $[n]$, with domain of size $r$, such that $\overline{\pi}$ does not conflict with $\sigma$. Then $\h(\sigma,\overline{\rho},\overline{\pi})$ is $(1,2\alpha)$-quasiregular, and $\mu(\h(\sigma,\overline{\rho},\overline{\pi})) \geq \tfrac{1}{2}\mu(\h(\sigma,\overline{\rho}))$.
\end{claim}

\begin{proof}
For brevity, let us write $\h' : = \h(\sigma,\overline{\rho})$. Since $\h'$ is $(1,\alpha)$-quasiregular, we have
$$\mu(\h') \leq \mu(\h'(\overline{\pi})) + \alpha r\mu(\h')/(n-b).$$
Rearranging, we have
\begin{equation}\label{eq:compare-measures-1}\mu(\h'(\overline{\pi})) \geq (1-\tfrac{\alpha r}{n-b}) \mu(\h') \geq \tfrac{3}{4}\mu(\h'),\end{equation}
proving the second statement of the claim. Now suppose for a contradiction that there exist $i \in [n] \setminus \Domain(\sigma)$ and $j \in [n]\setminus \Range(\sigma)$ such that $\mu(\h'(\overline{\pi},i \mapsto j)) \geq 2\alpha \mu(\h'(\overline{\pi}))$. Using (\ref{eq:compare-measures-1}), it follows that
$$\mu(\h'(\overline{\pi},i \mapsto j)) \geq 2\alpha\cdot \tfrac{3}{4} \mu(\h')= \tfrac{3}{2}\alpha\mu(\h').$$
On the other hand,
$$\mu(\h'(i \mapsto j)) \geq (1-\tfrac{r}{n-b})\mu(\h'(\overline{\pi},i \mapsto j)),$$
so
$$\mu(\h'(i \mapsto j)) \geq (1-\tfrac{r}{n-b})\tfrac{3}{2}\alpha\mu(\h') > \alpha \mu(\h'),$$
a contradiction.
\end{proof}

\begin{claim}
\label{claim:quasi-small-error}
If $b,n,r,s \in \mathbb{N}$, $1 \leq \alpha \leq 2$ and $\h(\sigma,\overline{\rho}) \subset S_n(\sigma,\overline{\rho})$ is $(s,\alpha)$-quasiregular, where $|\Domain(\sigma)|+|\Domain(\rho)|\leq b$, and $\pi$ is a bijection such that $\overline{\pi}$ does not conflict with $\sigma$ and has $|\Domain(\pi)|=r$, where $n \geq r^2+b$, then $\h(\sigma,\overline{\rho},\overline{\pi})$ is $(s,(1+\tfrac{4}{\sqrt{n-b}})\alpha)$-quasiregular, and $\mu(\h(\sigma,\overline{\rho},\overline{\pi})) \geq (1-\tfrac{2}{\sqrt{n-b}})\mu(\h(\sigma,\overline{\rho}))$.
\end{claim}
\begin{proof}
As above, let us write $\h' : = \h(\sigma,\overline{\rho})$. As above, we have
\begin{equation}\label{eq:compare-measures}\mu(\h') \leq \mu(\h'(\overline{\pi})) + \alpha r\mu(\h')/(n-b),\end{equation}
and therefore
$$\mu(\h'(\overline{\pi})) \geq (1-\tfrac{\alpha r}{n-b}) \mu(\h') \geq (1-\tfrac{2}{\sqrt{n-b}})\mu(\h'),$$
proving the second statement of the claim. Now suppose for a contradiction that there exists a bijection $\tau$ between sets of size $s$, such that $\mu(\h'(\overline{\pi},\tau)) > (1+\delta)\alpha \mu(\h'(\overline{\pi}))$, where $\delta: = 4/\sqrt{n-b}$. Using (\ref{eq:compare-measures}), it follows that
$$\mu(\h'(\overline{\pi},\tau)) > (1+\delta)(1-\tfrac{2}{\sqrt{n-b}})\alpha\mu(\h').$$
On the other hand,
$$\mu(\h'(\tau)) \geq (1-\tfrac{r}{n-b})\mu(\h'(\overline{\pi},\tau)) \geq (1-\tfrac{1}{\sqrt{n-b}})\mu(\h'(\overline{\pi},\tau)),$$
so
$$\mu(\h'(\tau)) \geq (1+\delta)(1-\tfrac{1}{\sqrt{n-b}})(1-\tfrac{2}{\sqrt{n-b}})\alpha\mu(\h') =  (1+\delta)(1-\tfrac{\delta}{4})(1-\tfrac{\delta}{2})\alpha\mu(\h')> \alpha \mu(\h'),$$
a contradiction.
\end{proof}

The following lemma will enable us, when starting from a pair of polynomially dense, $(s,n^{-r})$-uncaptureable families, to `bootstrap' this weak notion of pseduorandomness ($(s,n^{-r})$-uncaptureability) into the stronger notion of pseudorandomness, viz., $(1,\Theta(\sqrt{n})$)-quasiregularity, at the cost of passing to a pair of (reasonably large) subfamilies. The proof is combinatorial, using a density increment argument.

\begin{lem}
\label{lem:uncap-quasi}
Let $b,n,r,s \in \mathbb{N}$ with
$$n \geq 8r\sqrt{n}+2r+b,\quad s \geq 2r-1,$$
and let $\f_1(\pi_1,\overline{\sigma_1}) \subset S_n(\pi_1,\overline{\sigma_1}),\ \f_2(\pi_2,\overline{\sigma_2}) \subset S_n(\pi_2,\overline{\sigma_2})$ be $(s,n^{-r})$-uncaptureable with $|\Domain(\pi_i)|+|\Domain(\sigma_i)| \leq b$ for $i=1,2$. Suppose that for all $x,y \in [n]$, whenever $\pi_1(x)=y$ with $x \notin \Domain(\pi_2)$ and $y \notin \Range(\pi_2)$, we have $\sigma_2(x)=y$, and whenever $\pi_2(x)=y$ with $x \notin \Domain(\pi_1)$ and $y \notin \Range(\pi_1)$, we have $\sigma_1(x)=y$. Then there exist bijections $\pi_3,\pi_4$ bijecting between sets of size less than $2r$, such that $\f_{1}\left(\pi_{1},\overline{\sigma_1},\pi_3,\overline{\pi_{4}}\right),\f_{2}\left(\pi_2,\overline{\sigma_2},\overline{\pi_{3}},\pi_{4}\right)$
are both $\left(1,2\sqrt{n}\right)$-quasiregular, with measures greater than $\tfrac{1}{2}n^{-r}$. Moreover, $\pi_1$ and $\pi_3$ have disjoint domains and disjoint ranges, $\pi_2$ and $\pi_4$ have disjoint domains and disjoint ranges, and if $\pi_1$ and $\pi_2$ agree in exactly $u$ places, then the same is true of the bijections $\pi_1 \cup \pi_3$ and $\pi_2 \cup \pi_4$.
\end{lem}
\begin{proof}
Set $\alpha := \sqrt{n}$. Let $\f_1(\pi_1,\overline{\sigma_1})$ and $\f_2(\pi_2,\overline{\sigma_2})$ be $(s,n^{-r})$-uncaptureable. Then, in particular, we have $\min\{\mu(\f_1(\pi_1,\overline{\sigma_1})),\mu(\f_2(\pi_2,\overline{\sigma_2}))\} > n^{-r}$. Suppose that $\f_1(\pi_1,\overline{\sigma_1})$ and $\f_2(\pi_2,\overline{\sigma_2})$ are not both $(1,\alpha)$-quasiregular. We may assume that $\f_1(\pi_1,\overline{\sigma_1})$ is not $\left(1,\alpha\right)$-quasiregular. Then there exist $x \in [n]\setminus \Domain(\pi_1)$ and $y \in [n]\setminus \Range(\pi_1)$ such that $\mu(\f_1(\pi_1,\overline{\sigma_1},x \mapsto y)) \geq \alpha \mu(\f_1(\pi_1,\overline{\sigma_1})) > \alpha n^{-r}$. Note that $\pi_2$ does not map $x$ to $y$. (Indeed, we have $x \notin \Domain(\pi_1)$ and $y \notin \Range(\pi_1)$, so if $\pi_2(x)=y$, then by hypothesis we would have $\sigma_1(x)=y$, a contradiction.) Repeat this process (first with the bijection $\pi_1 \cup (x \mapsto y)$ in place of $\pi_1$, and so on) until we obtain a bijection $\pi_3$ such that $\f_1(\pi_1,\overline{\sigma_1},\pi_3)$ is $(1,\alpha)$-quasiregular (clearly, this happens after less than $2r$ steps, i.e.\ when the bijection $\pi_3$ has domain of cardinality less than $2r$, since
$$\mu(\f_1(\pi_1,\overline{\sigma_1},\pi_3)) > \alpha^{|\Domain(\pi_3)|} \mu(\f_1(\pi_1,\overline{\sigma_1})) > (\sqrt{n})^{|\Domain(\pi_3)|} n^{-r} \geq 1$$
if $|\Domain(\pi_3)| \geq 2r$). Hence, there exists a bijection $\pi_3$ with domain of cardinality less than $2r$, such that $\f_1(\pi_1,\overline{\sigma_1},\pi_3)$ is $(1,\alpha)$-quasiregular, $\mu(\f_1(\pi_1,\overline{\sigma_1},\pi_3))>n^{-r}$, and $\overline{\pi_3}$ does not conflict anywhere with $\pi_2$ (i.e., if $\pi_2(z)=w$ then $\pi_3(z) \neq w$). Provided $s \geq 2r-1$, since $\f_2(\pi_2,\overline{\sigma_2})$ is $(s,n^{-r})$-uncaptureable (and $\overline{\pi_3}$ does not conflict with $\pi_2$), we have $\mu(\f_2(\pi_2,\overline{\sigma_2},\overline{\pi_3})) > n^{-r}$. If $\f_2(\pi_2,\overline{\sigma_2},\overline{\pi_3})$ is $(1,\alpha)$-quasiregular, then we are done. If not, there exist $x \in [n] \setminus \Domain(\pi_2)$ and $y \in [n] \setminus \Range(\pi_2)$ such that $\mu(\f_2(\pi_2,\overline{\sigma_2},\overline{\pi_3},x \mapsto y)) \geq \alpha \mu(\f_2(\pi_2,\overline{\sigma_2},\overline{\pi_3})) > \alpha n^{-r}$. Note that $\pi_1$ does not map $x$ to $y$, by the same argument as above. Repeat this process until we obtain a bijection $\pi_4$ such that $\f_2(\pi_2,\overline{\sigma_2},\overline{\pi_3},\pi_4)$ is $(1,\alpha)$-quasiregular; by the same argument as above, we have $|\Domain(\pi_4)| < 2r$. Since $\f_1(\pi_1,\overline{\sigma_1},\pi_3)$ is $(1,\alpha)$-quasiregular (and $\overline{\pi_4}$ does not conflict with $\pi_1 \cup \pi_3)$, Claim \ref{claim:quasi} implies that $\f_1(\pi_1,\overline{\sigma_1},\pi_3,\overline{\pi_4})$ is $(1,2\alpha)$-quasiregular, and that
$$\mu(\f_1(\pi_1,\overline{\sigma_1},\pi_3,\overline{\pi_4})) \geq \tfrac{1}{2} \mu(\f_1(\pi_1,\overline{\sigma_1},\pi_3)) > \tfrac{1}{2}n^{-r}.$$
Since
$$\f_{1}\left(\pi_{1},\overline{\sigma_1},\pi_3,\overline{\pi_{4}}\right)\neq \emptyset,\quad \f_{2}\left(\pi_2,\overline{\sigma_2},\overline{\pi_{3}},\pi_{4}\right) \neq \emptyset,$$
we must have $\pi_3 \cap \pi_2 = \pi_3 \cap \pi_4 = \pi_4 \cap \pi_1 = \emptyset$, and therefore $|(\pi_1 \cup \pi_3) \cap (\pi_2 \cup \pi_4)| = |\pi_1 \cap \pi_2|$, as required. This completes the proof of the lemma.
\end{proof}

Next, we observe that somewhat quasiregular famlies are highly uncaptureable.

\begin{claim}
\label{claim:quasi-uncap}
Let $\beta >1$, let $b,N \in \mathbb{N}$ and let $\delta >0$. If $\h(\sigma,\overline{\rho}) \subset S_n(\sigma,\overline{\rho})$ is $(1,\beta)$-quasiregular, where $|\Domain(\sigma)|+|\Domain(\rho)|\leq b$, and $\mu(\h(\sigma,\overline{\rho})) > \delta$, then $\h(\sigma,\overline{\rho})$ is $(N,\tfrac{1}{2}\delta)$-uncaptureable, provided $N \leq (n-b)/(2\beta)$.
\end{claim}
\begin{proof}
Write $\h' = \h(\sigma,\overline{\rho})$. Suppose for a contradiction that $\h'$ is $(N,\tfrac{1}{2}\delta)$-captureable. Then there exist sets $S\subset [n]\setminus \Domain(\sigma)$ and $T\subset [n]\setminus \Range(\sigma)$ with $|S|=|T| \leq N$, and a bijection $\pi:S \to T$, such that $\mu(\h'(\overline{\pi})) \leq \tfrac{1}{2}\delta$. It follows that
$$\mu(\h') \leq \mu(\h'(\overline{\pi}))+\beta N \mu(\h')/(n-b) \leq \tfrac{1}{2} \delta + \beta N \mu(\h')/(n-b),$$
and therefore
$$\mu(\h') \leq \tfrac{1}{2} \delta / (1-\beta N/(n-b)) \leq \delta,$$
a contradiction.
\end{proof}

Our second `bootstrapping' lemma enables us to find a pair of dense, highly quasiregular families within a pair of dense, highly uncaptureable families. The proof is very similar to that of Lemma \ref{lem:uncap-quasi}.

\begin{lem}
\label{lem:uncap-very-quasi}
Let $c>0$, let $b,n,N \in \mathbb{N}$, let $0 < \epsilon <1$, let $r,s \in \mathbb{N}$, and let $\f_1(\pi_1,\overline{\sigma_1}) \subset S_n(\pi_1,\overline{\sigma_1})$, $\f_2(\pi_2,\overline{\sigma_2}) \subset S_n(\pi_2,\overline{\sigma_2})$ be $(N,cn^{-r})$-uncaptureable with $|\Domain(\pi_i)|+|\Domain(\sigma_i)| \leq b$ for $i=1,2$. Suppose that for all $x,y \in [n]$, whenever $\pi_1(x)=y$ with $x \notin \Domain(\pi_2)$ and $y \notin \Range(\pi_2)$, we have $\sigma_2(x)=y$, and whenever $\pi_2(x)=y$ with $x \notin \Domain(\pi_1)$ and $y \notin \Range(\pi_1)$, we have $\sigma_1(x)=y$. Suppose further that $N \geq (r \log n - \log c)s/\log(1+\epsilon) =:b'$, that $\epsilon \geq 8/\sqrt{n-b-b'}$ and that $n \geq (b')^2+b'+b$. Then there exist bijections $\pi_3,\pi_4$ bijecting between sets of size less than $b'$, such that
$$\f_{1}\left(\pi_{1},\overline{\sigma_1},\pi_3,\overline{\pi_{4}}\right),\f_{2}\left(\pi_2,\overline{\sigma_2},\overline{\pi_{3}},\pi_{4}\right)$$
are both $\left(s,1+2\epsilon\right)$-quasiregular, with measures greater than $\tfrac{1}{2}cn^{-r}$. Moreover, $\pi_1$ and $\pi_3$ have disjoint domains and disjoint ranges, $\pi_2$ and $\pi_4$ have disjoint domains and disjoint ranges, and if $\pi_1$ and $\pi_2$ agree in exactly $u$ places, then the same is true of the bijections $\pi_1 \cup \pi_3$ and $\pi_2 \cup \pi_4$.
\end{lem}

\begin{proof}
Let $\f_1(\pi_1,\overline{\sigma_1})$ and $\f_2(\pi_2,\overline{\sigma_2})$ be $(N,cn^{-r})$-uncaptureable. Then, in particular, we have
$$\min\{\mu(\f_1(\pi_1,\overline{\sigma_1})),\mu(\f_2(\pi_2,\overline{\sigma_2}))\} > cn^{-r}.$$
Suppose that $\f_1(\pi_1,\overline{\sigma_1})$ and $\f_2(\pi_2,\overline{\sigma_2})$ are not both $(s,1+\epsilon)$-quasiregular. We may assume that $\f_1(\pi_1,\overline{\sigma_1})$ is not $(s,1+\epsilon)$-quasiregular. Then there exist sets $S \subset [n] \setminus \Domain(\pi_1)$ and $T \subset [n] \setminus \Range(\pi_1)$ with $|S|=|T|=s$, and a bijection $\pi:S\to T$, such that $\mu(\f_1(\pi_1,\overline{\sigma_1},\pi)) \geq (1+\epsilon)\mu(\f_1(\pi_1,\overline{\sigma_1})) > (1+\epsilon)cn^{-r}$. Note that $\pi$ does not agree anywhere with $\pi_2$. (Indeed, if $\pi(x)=y$, then $x \notin \Domain(\pi_1)$ and $y \notin \Range(\pi_1)$, so by hypothesis $\sigma_2(x)=y$, and therefore we cannot have $\pi_2(x)=y$.) Repeat this process (first with the bijection $\pi_1 \cup \pi$ in place of $\pi_1$, and so on) until we obtain a bijection $\pi_3$ such that $\f_1(\pi_1,\overline{\sigma_1},\pi_3)$ is $(s,1+\epsilon)$-quasiregular. This happens after less than
$$M:=\frac{r \log n - \log c}{\log(1+\epsilon)}$$
steps, i.e.\ when the domain of the bijection $\pi_3$ has cardinality less than $Ms=:b'$, since after $\lceil M \rceil$ steps, we would obtain a bijection $\pi'$ with
$$\mu(\f_1(\pi_1,\overline{\sigma_1},\pi')) \geq (1+\epsilon)^{M} \mu(\f_1(\pi_1,\overline{\sigma_1})) > (1+\epsilon)^{M} cn^{-r} = 1.$$
Hence, there exists a bijection $\pi_3$ with domain of cardinality less than $Ms$, such that $\f_1(\pi_1,\overline{\sigma_1},\pi_3)$ is $(s,1+\epsilon)$-quasiregular, $\mu(\f_1(\pi_1,\overline{\sigma_1},\pi_3))>cn^{-r}$, and $\overline{\pi_3}$ does not conflict with $\pi_2$. Provided $N \geq b'$, since $\f_2(\pi_2,\overline{\sigma_2})$ is $(N,cn^{-r})$-uncaptureable (and $\overline{\pi_3}$ does not conflict with $\pi_2$), we have $\mu(\f_2(\pi_2,\overline{\sigma_2},\overline{\pi_3})) > cn^{-r}$. If $\f_2(\pi_2,\overline{\sigma_2},\overline{\pi_3})$ is $(s,1+\epsilon)$-quasiregular, then we are done. If not, there exist sets $S \subset [n] \setminus \Domain(\pi_2)$ and $T \subset [n] \setminus \Range(\pi_2)$ with $|S|=|T|=s$, and a bijection $\pi:S\to T$, such that $\mu(\f_2(\pi_2,\overline{\sigma_2},\overline{\pi_3},\pi)) \geq (1+\epsilon) \mu(\f_2(\pi_2,\overline{\sigma_2},\overline{\pi_3})) > (1+\epsilon)cn^{-r}$. Note that $\pi$ does not agree anywhere with $\pi_1$, by the same argument as above. Repeat this process until we obtain a bijection $\pi_4$ such that $\f_2(\pi_2,\overline{\sigma_2},\overline{\pi_3},\pi_4)$ is $(s,1+\epsilon)$-quasiregular; by the same argument as above, we have $|\Domain(\pi_4)| < Ms$, where $M$ is as above. Observe that
$$|\Domain(\pi_1)|+|\Domain(\sigma_1)|+|\Domain(\pi_3)| \leq b+b'.$$
Since $\f_1(\pi_1,\overline{\sigma_1},\pi_3)$ is $(s,1+\epsilon)$-quasiregular (and $\overline{\pi_4}$ does not conflict with $\pi_1 \cup \pi_3)$, Claim \ref{claim:quasi-small-error} implies that $\f_1(\pi_1,\overline{\sigma_1},\pi_3,\overline{\pi_4})$ is $(s,1+2\epsilon)$-quasiregular, and that
$$\mu(\f_1(\pi_1,\overline{\sigma_1},\pi_3,\overline{\pi_4})) \geq \tfrac{1}{2} \mu(\f_1(\pi_1,\overline{\sigma_1},\pi_3) > \tfrac{1}{2}cn^{-r},$$
provided $\epsilon \geq 8/\sqrt{n-b-b'}$ and $n \geq (b')^2+b'+b$.

Since
$$\f_{1}\left(\pi_{1},\overline{\sigma_1},\pi_3,\overline{\pi_{4}}\right)\neq \emptyset,\quad \f_{2}\left(\pi_2,\overline{\sigma_2},\overline{\pi_{3}},\pi_{4}\right) \neq \emptyset,$$
we must have $\pi_3 \cap \pi_2 = \pi_3 \cap \pi_4 = \pi_4 \cap \pi_1 = \emptyset$, and therefore $|(\pi_1 \cup \pi_3) \cap (\pi_2 \cup \pi_4)| = |\pi_1 \cap \pi_2|$, as required. This completes the proof of the lemma.
\end{proof}

Given a pair of dense, highly quasiregular families $\h_1(\sigma_1,\overline{\rho_1}),\h_2(\sigma_2,\overline{\rho_2})$, the following lemma will enable us to increase the number of agreements between $\sigma_1$ and $\sigma_2$ (up to $t-1$), without seriously affecting the densities or the quasiregularity. (This, in turn, will enable us to reduce our task to finding, in a pair of dense, highly quasiregular families, a pair of permutations that disagree everywhere, hence reducing the algebraic part of the proof to the case where $t=1$, making the calculations significantly cleaner.) The proof is probabilistic, via a random sampling argument.

\begin{prop}
\label{prop:extra-agreements}
Let $t \leq s$, let $0 \leq \epsilon < \tfrac{3}{32}$, suppose $|\Domain(\sigma_i)|,|\Domain(\rho_i)| \leq b$ for $i=1,2$, and suppose that $\h_1(\sigma_1,\overline{\rho_1}) \subset S_n(\sigma_1,\overline{\rho_1})$ and $\h_2(\sigma_2,\overline{\rho_2}) \subset S_n(\sigma_2,\overline{\rho_2})$ are $(s,1+\epsilon)$-quasiregular. Provided $\epsilon \geq c_0bt/n$ for some absolute constant $c_0$, there exists a bijection $\pi$ between sets of size $t$, with domain disjoint from the domains of $\sigma_i$ and of $\rho_i$ and ranges disjoint from the ranges of $\sigma_i$ and of $\pi_i$ (for $i=1,2$), such that
\begin{align*}
\mu(\h_1(\sigma_1,\overline{\rho_1},\pi)) & \geq (1-4\epsilon)\mu(\h_1(\sigma_1,\overline{\rho_1})),\\
\mu(\h_2(\sigma_2,\overline{\rho_2},\pi)) & \geq (1-4\epsilon)\mu(\h_2(\sigma_2,\overline{\rho_2})),
\end{align*}
and $\h_i(\sigma_i,\overline{\rho_i},\pi)$ is $(s-t,1+8\epsilon)$-quasiregular for each $i\in\{1,2\}$.
\end{prop}
\begin{proof}
Fix a set $S_0 \in [n]^{(t)}$ such that $S_0 \cap (\Domain(\sigma_1) \cup \Domain(\sigma_2) \cup \Domain(\rho_1) \cup \Domain(\rho_2)) = \emptyset$. Let $\mathcal{V}$ denote the set of all bijections $\pi$ such that $\Domain(\pi) = S_0$ and $\Range(\pi) \in [n]^{(t)}$. Let us say that a bijection in $\mathcal{V}$ is {\em good} if its range is disjoint from the ranges of $\sigma_i$ and of $\rho_i$ (for $i=1,2$); otherwise, let us say it is {\em bad}. Let $\mathcal{U}$ be the set of all good bijections in $\mathcal{V}$. Observe that if $i \in \{1,2\}$, and $\pi$ is chosen according to the probability distribution $\mathcal{D}_i$ defined by
$$\Pr_{\pi \sim \mathcal{D}_i}[\pi] = \frac{|S_n(\sigma_i,\overline{\rho_i},\pi)|}{|S_n(\sigma_i,\overline{\rho_i})|}\quad \forall \pi \in \mathcal{V};$$
then $\Pr_{\pi \sim \mathcal{D}_i}[\pi \in \mathcal{U}] = 1-O(bt/n)$, since for any $j \in \Range(\rho_i)$, we have $\Pr_{\pi \sim \mathcal{D}_i} [j \in \Range(\pi)] = O(t/n)$. Moreover, conditional on the event $\{\pi \in \mathcal{U}\}$, the conditional distribution of $\pi \sim \mathcal{D}_i$ is uniform on $\mathcal{U}$ (for $i=1,2$). We claim that
$$\Pr_{\pi \in \mathcal{U}} [\mu(\h_i(\sigma_i,\overline{\rho_i},\pi)) \leq (1-4\epsilon)\mu(\h_i(\sigma_i,\overline{\rho_i}))] < \tfrac{1}{4}$$
for $i=1,2$ (where the probability $\Pr_{\pi \in \mathcal{U}}$ denotes the probability with respect to the uniform distribution on $\mathcal{U}$). Indeed, suppose for a contradiction that
$$\Pr_{\pi \in \mathcal{U}} [\mu(\h_1(\sigma_1,\overline{\rho_1},\pi)) \leq (1-4\epsilon)\mu(\h_1(\sigma_1,\overline{\rho_1}))] \geq \tfrac{1}{4}.$$
Then
\begin{equation}\label{eq:large-prob} \Pr_{\pi \sim \mathcal{D}_1}[\mu(\h_1(\sigma_1,\overline{\rho_1},\pi)) \leq (1-4\epsilon)\mu(\h_1(\sigma_1,\overline{\rho_1}))] \geq (1-O(bt/n))\cdot \tfrac{1}{4},\end{equation}
since $\Pr_{\pi \sim \mathcal{D}_1}[\pi \in \mathcal{U}] = 1-O(bt/n)$, and conditional on the event $\{\pi \in \mathcal{U}\}$, the conditional distribution of $\pi \sim \mathcal{D}_1$ is uniform on $\mathcal{U}$. But by the quasiregularity hypothesis, for {\em any} $\pi \in \mathcal{V}$ we have
\begin{equation}\label{eq:uniform-bound} \mu(\h_1(\sigma_1,\overline{\rho_1},\pi)) \leq (1+\epsilon)\mu(\h_1(\sigma_1,\overline{\rho_1})).\end{equation}
Combining (\ref{eq:large-prob}) with (\ref{eq:uniform-bound}) yields
\begin{align*} \mu(\h_1(\sigma_1,\overline{\rho_1})) & = \mathbb{E}_{\pi \sim \mathcal{D}_1}[\mu(\h_1(\sigma_1,\overline{\rho_1},\pi))] \\
& \leq (1/4 - O(bt/n))(1-4\epsilon)\mu(\h_1(\sigma_1,\overline{\rho_1}))\\
& \quad \, + (3/4+O(bt/n))(1+\epsilon)\mu(\h_1(\sigma_1,\overline{\rho_1}))\\
& = (1-\epsilon/4+O(bt/n))\mu(\h_1(\sigma_1,\overline{\rho_1}))\\
& < \mu(\h_1(\sigma_1,\overline{\rho_1})),
\end{align*}
provided $\epsilon \geq c_0bt/n$ for some absolute constant $c_0$, a contradiction. It follows that
$$\Pr_{\pi \in \mathcal{U}} [\mu(\h_1(\sigma_1,\overline{\rho_1},\pi)) \leq (1-4\epsilon)\mu(\h_1(\sigma_1,\overline{\rho_1}))] < \tfrac{1}{4},$$
and similarly,
$$\Pr_{\pi \in \mathcal{U}} [\mu(\h_2(\sigma_2,\overline{\rho_2},\pi)) \leq (1-4\epsilon)\mu(\h_2(\sigma_2,\overline{\rho_2}))] < \tfrac{1}{4}.$$
Hence, by the union bound, we have
$$\Pr_{\pi \in \mathcal{U}} [\mu(\h_i(\sigma_i,\overline{\rho_i},\pi)) > (1-4\epsilon)\mu(\h_i(\sigma_i,\overline{\rho_i})) \ \forall i \in \{1,2\}] > \tfrac{1}{2}.$$
Hence, there exists a bijection $\pi$ such that the first two conditions of the Proposition hold. Fix such a bijection $\pi$. Suppose for a contradiction that $\h_1(\sigma_1,\overline{\rho_1},\pi))$ is not $(s-t,1+8\epsilon)$-quasiregular. Then there exists a bijection $\pi'$ between sets of size $s-t$, such that
\begin{align*} \mu(\h_1(\sigma_1,\overline{\rho_1},\pi,\pi')) & \geq (1+8\epsilon) \mu(\h_1(\sigma_1,\overline{\rho_1},\pi))\\
& \geq (1+8\epsilon)(1-4\epsilon)\mu(\h_1(\sigma_1,\overline{\rho_1}))\\
& > (1+\epsilon)\mu(\h_1(\sigma_1,\overline{\rho_1})),
\end{align*}
contradicting the quasiregularity of $\h_1(\sigma_1,\overline{\rho_1})$ (by considering the bijection $\pi \cup \pi'$). Hence, $\h_1(\sigma_1,\overline{\rho_1},\pi))$ is $(s-t,1+8\epsilon)$-quasiregular. The same argument shows that $\h_2(\sigma_2,\overline{\rho_2},\pi))$ is $(s-t,1+8\epsilon)$-quasiregular, completing the proof.
\end{proof}

\section{Combinatorial and algebraic quasirandomness}
\label{sec:quasi}
We first define a new combinatorial notion of pseudorandomness, one which is closely related to quasiregularity, and which will be more closely related than quasiregularity to our algebraic notion of pseudorandomness.
\begin{defn}
Let $f:S_n \to \mathbb{R}$. We say that $f$ is {\em $(r,\epsilon)$-quasirandom} if
\begin{equation}\label{eq:quasirandom} \mathbb{E}_{\pi} (\mathbb{E}[f(\pi)]-\mathbb{E}[f])^2 \leq \epsilon (\mathbb{E}[f])^2,\end{equation}
where the expectation on the left is over a uniform random bijection $\pi$ between $r$-element subsets of $[n]$.
\end{defn}

Note also that the left-hand side of (\ref{eq:quasirandom}) is precisely $\Var_{\pi}(\mathbb{E}[f(\pi)])$. Informally speaking, a function is highly quasirandom if, on average, restricting the function to the set of all permutations agreeing with a uniform random bijection with domain of fixed size, does not change the expectation of the function very much.

We remark that if $f:S_n \to \mathbb{R}$ is $(\epsilon,r)$-quasirandom, then it is $(\epsilon,s)$-quasirandom for all $s \leq r$, by the Cauchy-Schwarz inequality.

The following easy lemma says that highly quasiregular functions are also highly quasirandom.

\begin{lem}
\label{lem:quasireg-quasirandom}
If $0\leq \epsilon <1$ and $f:S_n \to [0,1]$ is $(s,1+\epsilon)$-quasiregular, then $f$ is $(s,2\epsilon+\epsilon^2)$-quasirandom.
\end{lem}
\begin{proof}
Suppose $f:S_n \to \mathbb{R}$ is $(s,1+\epsilon)$-quasiregular. Then
$$\mathbb{E}_{\pi}(\mathbb{E}[f(\pi)])^2 \leq (1+\epsilon)^2 (\mathbb{E}[f])^2$$
and therefore
$$\Var_{\pi}(\mathbb{E}[f(\pi)]) \leq (2\epsilon+\epsilon^2)(\mathbb{E}[f])^2,$$
as required. 
\end{proof}

We now introduce our algebraic notion of pseudorandomness.

\begin{defn}
If $r \in \mathbb{N}$ and $\epsilon >0$, we say that $f:S_n \to \mathbb{R}$ is {\em $(r,\epsilon)$-algebraically quasirandom} if for all $\alpha \vdash n$ with $\alpha \neq (n)$ and $\alpha_1 \geq n-r$, we have 
$$\|P_{\alpha}(f)\|_2^2 \leq \epsilon \mathsf{f}^{\alpha} (\mathbb{E}[f])^2.$$
\end{defn}

The next lemma, which is key, says that highly quasirandom functions are highly algebraically-quasirandom. The proof is analytic/algebraic.

\begin{lem}
\label{lem:comb-alg}
Let $r \in \mathbb{N}$ and $\epsilon >0$. If $f:S_n \to \mathbb{R}$ is $(r,\epsilon)$-quasirandom, then it is $(r,O_r(\epsilon))$-algebraically quasirandom.
\end{lem}
\begin{proof}
Let $f:S_n \to \mathbb{R}$ be $(\epsilon,r)$-quasirandom. Write $g: = f-\mathbb{E}[f]$. Let $\alpha \vdash n$ with $\alpha \neq (n)$ and $\alpha_1 \geq n-r$. Let
$$\chi_{\alpha} = \sum_{\beta \geq \alpha} c_{\beta}\xi_{\beta}$$
be the expression of $\chi_{\alpha}$ as a linear combination of the permutation characters, as guaranteed by the determintantal formula. As remarked above, in this case, the sum in the determinantal formula has at most $(r+1)!$ non-zero summands, and therefore
$$\sum_{\beta \geq \alpha} |c_{\beta}| \leq (r+1)! = O_r(1).$$
Using (\ref{eq:inner-prod}), we have
\begin{align*}
\|P_{\alpha}(f)\|_2^2 &= \|P_{\alpha}(g)\|_2^2\\
& = \frac{\mathsf{f}^{\alpha}}{(n!)^2} \sum_{\sigma,\pi \in S_n} g(\sigma)g(\pi) \chi_{\alpha}(\sigma \pi^{-1})\\
& = \frac{\mathsf{f}^{\alpha}}{(n!)^2} \sum_{\beta \geq \alpha} c_{\beta} \sum_{\sigma,\pi \in S_n} g(\sigma)g(\pi) \xi_{\beta}(\sigma \pi^{-1})\\
& =  \frac{\mathsf{f}^{\alpha}}{(n!)^2} \sum_{\beta \geq \alpha} c_{\beta} \sum_{\sigma,\pi \in S_n} g(\sigma)g(\pi) \sum_{T \in \t_{\beta}}1_{\{\sigma \pi^{-1}(T)=T\}}\\
& = \frac{\mathsf{f}^{\alpha}}{(n!)^2} \sum_{\beta \geq \alpha} c_{\beta} \sum_{T,T' \in \t_\beta} \sum_{\sigma,\pi \in S_n} g(\sigma) 1_{\{\sigma(T)=T'\}} g(\pi) 1_{\{\pi(T)=T'\}}\\
&= \frac{\mathsf{f}^{\alpha}}{(n!)^2} \sum_{\beta \geq \alpha} c_{\beta} \sum_{T,T' \in \t_\beta} \left(\sum_{\sigma \in S_n} g(\sigma) 1_{\{\sigma(T)=T'\}}\right)^2\\
& = \mathsf{f}^{\alpha} \sum_{\beta \geq \alpha} c_{\beta} \left(\frac{s(\beta)}{n!}\right)^2 \sum_{T,T' \in \t_\beta} \left(\frac{1}{s(\beta)}\sum_{\sigma \in S_n} g(\sigma) 1_{\{\sigma(T)=T'\}}\right)^2\\
& = \mathsf{f}^{\alpha} \sum_{\beta \geq \alpha} c_{\beta} \left(\frac{s(\beta)}{n!}\right)^2 \cdot |\t_\beta|^2 \cdot \frac{1}{|\t_\beta|^2} \sum_{T,T' \in \t_\beta} \left(\frac{1}{s(\beta)}\sum_{\sigma \in S_n} g(\sigma) 1_{\{\sigma(T)=T'\}}\right)^2\\
& \leq \mathsf{f}^{\alpha} \sum_{\beta \geq \alpha} c_{\beta} \left(\frac{s(\beta)}{n!}\right)^2 |\t_\beta|^2 \epsilon (\mathbb{E}[f])^2\\
& = \mathsf{f}^{\alpha}\epsilon (\mathbb{E}[f])^2 \sum_{\beta \geq \alpha} c_{\beta} \left(\frac{s(\beta)|\t_\beta| }{n!}\right)^2\\
& = \epsilon (\mathbb{E}[f])^2 \mathsf{f}^{\alpha} \sum_{\beta \geq \alpha} c_{\beta}\\
& \leq \epsilon (\mathbb{E}[f])^2 \mathsf{f}^{\alpha} \sum_{\beta \geq \alpha} |c_{\beta}|\\
& = O_r(1) \epsilon (\mathbb{E}[f])^2 \mathsf{f}^{\alpha},
\end{align*}
using the fact that $s(\beta)|\t_\beta| =n!$, by the orbit-stabilizer theorem. The first inequality above uses Cauchy-Schwarz, combined with the fact that
$$\frac{1}{s(\beta)}\sum_{\sigma \in S_n} g(\sigma) 1_{\{\sigma(T)=T'\}}$$
is simply an average of quantities of the form
$$\mathbb{E}[f(\pi)] - \mathbb{E}[f],$$
with the $\pi$'s corresponding to the possible choices of images under $\sigma$ of numbers in rows of $T$ below the first row.
\end{proof}

The following lemma encapsulates the spectral part of our argument; it says that two highly quasirandom fractional families must either put non-zero weights on a pair of permutations that disagree everywhere, or else the product of their measures must be small.

\begin{lem}
\label{lem:fourier}
For each $r \in \mathbb{N}$, there exists $\epsilon_0(r)>0$ such that the following holds. Let $f_1,f_2:S_n \to [0,1]$ such that each $f_i$ is $(r,\epsilon)$-quasirandom, where $\epsilon \leq \epsilon_0(r)$. Suppose that $f_1(\sigma_1)f_2(\sigma_2) = 0$ whenever $\sigma_1 \cap \sigma_2=\emptyset$. Then
$$\sqrt{\mathbb{E}[f_1] \mathbb{E}[f_2]} = O_r(n^{-r-1}).$$
\end{lem}
\begin{proof}
Let $f_1,f_2$ be as in the statement of the lemma, where $\epsilon_0(r)$ is to be chosen later. Let $A$ be the normalized adjacency matrix of the derangement graph on $S_n$, normalized so that all row and column sums are equal to 1. Then
\begin{align*} 0 & = \langle f_1,Af_2 \rangle\\
& = \sum_{\alpha \vdash n} \lambda_{\alpha} \langle P_{\alpha}(f_1),P_{\alpha}(f_2)\rangle\\
& = \mathbb{E}[f_1] \mathbb{E}[f_2] + \sum_{\alpha \in \mathcal{L}^*({\geq}n-r)} \lambda_{\alpha} \langle P_{\alpha}(f_1),P_{\alpha}(f_2)\rangle + \sum_{\alpha \in \mathcal{L}({<}n-r)} \lambda_{\alpha} \langle P_{\alpha}(f_1),P_{\alpha}(f_2)\rangle.
\end{align*}
By Corollary \ref{cor:low-evals}, we have
$$\max_{\alpha \in \mathcal{L}({<}n-r)} |\lambda_{\alpha}| = O_r(n^{-r-1}).$$
By Cauchy-Schwarz, we have
\begin{align*} \sum_{\alpha \vdash n} |\langle P_{\alpha}(f_1),P_{\alpha}(f_2)\rangle| & \leq \sum_{\alpha \vdash n} \|P_{\alpha}(f_1)\|_2 \|P_{\alpha}(f_2)\|_2\\
&\leq \sqrt{\sum_{\alpha \vdash n} \|P_{\alpha}(f_1)\|_2^2 } \sqrt{\sum_{\alpha \vdash n} \|P_{\alpha}(f_2)\|_2^2 }\\
& = \|f_1\|_2\|f_2\|_2\\
& \leq \sqrt{\mathbb{E}[f_1]\mathbb{E}[f_2]},
\end{align*}
and therefore
\begin{equation}
\label{eq:expansion}
0 = \mathbb{E}[f_1] \mathbb{E}[f_2] + \sum_{\alpha \in \mathcal{L}^{*}({\geq}n-r)} \lambda_{\alpha} \langle P_{\alpha}(f_1),P_{\alpha}(f_2)\rangle + O_r(n^{-r-1}) \sqrt{\mathbb{E}[f_1]\mathbb{E}[f_2]}.
\end{equation}
Since each $f_i$ is $(r,\epsilon)$-quasirandom, by Lemma \ref{lem:comb-alg} it is $(r,O_r(\epsilon))$-algebraically quasirandom, and therefore for each $\alpha \in \l^{*}({\geq}n-r)$, we have
$$\|P_{\alpha}(f_i)\|^2_2 = O_r(\epsilon) \mathsf{f}^{\alpha} (\mathbb{E}[f_i])^2,$$
and therefore, using Lemma \ref{lem:eval-estimate} and Cauchy-Schwarz again, we have
\begin{align*} \left|\sum_{\alpha \in \mathcal{L}^{*}({\geq}n-r)} \lambda_{\alpha} \langle P_{\alpha}(f_1),P_{\alpha}(f_2)\rangle \right| & \leq \sum_{\alpha \in \mathcal{L}^{*}({\geq}n-r)} |\lambda_{\alpha}| \cdot \|P_{\alpha}(f_1)\|_2 \|P_{\alpha}(f_1)\|_2\\
& \leq \sum_{\alpha \in \mathcal{L}^*({\geq}n-r)} O(1/\mathsf{f}^{\alpha}) \cdot  O_r(\epsilon) \mathsf{f}^{\alpha} \mathbb{E}[f_1]\mathbb{E}[f_2]\\
& = O_r(\epsilon) \mathbb{E}[f_1]\mathbb{E}[f_2],
\end{align*}
using the fact that $|\mathcal{L}_n^*({\geq}n-r)| = p(r)-1 = O_r(1)$ for all $n \geq 2r$, where $p(r)$ denotes the number of partitions of $r$. Substituting this into (\ref{eq:expansion}) yields
$$0 \geq (1-O_r(\epsilon)) \mathbb{E}[f_1] \mathbb{E}[f_2] + O_r(n^{-r-1}) \sqrt{\mathbb{E}[f_1] \mathbb{E}[f_2]},$$
and therefore, provided $\epsilon$ is sufficiently small depending on $r$, we have
$$\sqrt{\mathbb{E}[f_1] \mathbb{E}[f_2]} = O_r(n^{-r-1}),$$
as required.
\end{proof}

\section{Proof of our junta approximation result}
\label{sec:proof}
In this section, we use the tools developed above to prove Theorem \ref{thm:junta-approximation}. We may assume throughout that $n \geq C_0(r,t)$, for any fixed $C_0 = C_0(r,t) \in \mathbb{N}$ depending upon $r$ and $t$ alone.

It suffices to prove that if $\f \subset S_n$ is $(t-1)$-intersection-free, then the junta $\j = \langle J \rangle$ supplied by our regularity lemma (applied to $\f$, with $s=2r-1$) is $t$-intersecting (or equivalently, the set of bijections $J$ is $t$-intersecting). Our aim is to  prove that, if $\pi_1,\pi_2 \in J$ agree on less than $t$ points, then (using the uncaptureability of $\f(\pi_1),\f(\pi_2)$) we can find two permutations in $\f$ that agree with one another on exactly $t-1$ points, a contradiction. A slight complication arises in that we must compare pairs of uncaptureable families $\f(\pi_1),\f(\pi_2)$ for bijections $\pi_1,\pi_2$ whose domains (and ranges) may differ from one another. To deal with this complication, we produce $[0,1]$-valued functions $f,g:S_{\hat{n}} \to [0,1]$ (for some appropriate $\hat{n} \leq n$) such that:

\begin{itemize}
\item if there exist permutations $\sigma_1,\sigma_2 \in S_{\hat{n}}$ with $\sigma_1$ and $\sigma_2$ disagreeing everywhere and $f(\sigma_1)g(\sigma_2)>0$, then there exist permutations $\tau_i \in \f(\pi_i)$ (for $i=1,2$) such that $\tau_1$ and $\tau_2$ agree with one another on exactly $t-1$ points;
\item $f$ and $g$ are sufficiently quasiregular that there must exist permutations $\sigma_1,\sigma_2 \in S_{\hat{n}}$ with $\sigma_1$ and $\sigma_2$ disagreeing everywhere and $f(\sigma_1)g(\sigma_2)>0$.
\end{itemize}

Assume, for a contradiction, that there exist bijections $\pi_1,\pi_2 \in J$ such that $\pi_1$ and $\pi_2$ agree on $u<t$ points. Note that $\f(\pi_1)$ and $\f(\pi_2)$ are $(s,n^{-r})$-uncaptureable.

We first `preprocess' $\f(\pi_1)$ and $\f(\pi_2)$ as follows. Let us say that a pair $(x,y) \in [n]^2$ is {\em $\pi_1$-unique} if $\pi_1(x)=y$, $x \notin \Domain(\pi_{2})$ and $y \notin \Range(\pi_{2})$. Similarly, let us say that a pair $(x,y) \in [n]^2$ is {\em $\pi_2$-unique} if $\pi_2(x)=y$, $x \notin \Domain(\pi_{1})$ and $y \notin \Range(\pi_{1})$. Let $(x_1,y_1),\ldots,(x_k,y_k)$ denote the $\pi_1$-unique pairs. Now replace $\f(\pi_2)$ with $\f(\pi_2,\overline{\sigma_2})$, where $\sigma_2:\{x_1,\ldots,x_k\} \to \{y_1,\ldots,y_k\}$ is the bijection mapping $x_i$ to $y_i$ for each $i \in [k]$. Similarly, let $(v_1,w_1),\ldots,(v_l,w_l)$ denote the $\pi_2$-unique pairs. Now replace $\f(\pi_1)$ with $\f(\pi_1,\overline{\sigma_1})$, where $\sigma_1:\{v_1,\ldots,v_l\} \to \{w_1,\ldots,w_l\}$ is the bijection mapping $v_i$ to $w_i$ for each $i \in [l]$.

Next, we apply Lemma \ref{lem:uncap-quasi} to the above choice of $\f(\pi_1,\overline{\sigma_1}),\f(\pi_2,\overline{\sigma_2})$, with $s=2r-1$ and $b=2r-2$, obtaining two families $\f_{1}\left(\pi_{1},\overline{\sigma_1},\pi_3,\overline{\pi_{4}}\right),\f_{2}\left(\pi_2,\overline{\sigma_2},\overline{\pi_{3}},\pi_{4}\right)$
that are both $\left(1,2\sqrt{n}\right)$-quasiregular, with measures greater than $\tfrac{1}{2}n^{-r}$. (The `preprocessing' step above ensures that for all $x,y \in [n]$, whenever $\pi_1(x)=y$ with $x \notin \Domain(\pi_2)$ and $y \notin \Range(\pi_2)$, we have $\sigma_2(x)=y$, and whenever $\pi_2(x)=y$ with $x \notin \Domain(\pi_1)$ and $y \notin \Range(\pi_1)$, we have $\sigma_1(x)=y$.) Note that the bijection $\pi_1 \cup \pi_3$ agrees with the bijection $\pi_2 \cup \pi_4$ in exactly $u$ places.

By Claim \ref{claim:quasi-uncap}, applied with $\beta = 2\sqrt{n}$, $\delta = \tfrac{1}{2}n^{-r}$, $N = \lfloor \sqrt{n}/8\rfloor$, and $\h(\sigma,\overline{\rho}) = \f_{1}\left(\pi_{1},\overline{\sigma_1},\pi_3,\overline{\pi_{4}}\right)$ or $\f_{2}\left(\pi_2,\overline{\sigma_2},\overline{\pi_{3}},\pi_{4}\right)$, it follows that
\begin{equation} \label{eq:boosted-pair} \f_{1}\left(\pi_{1},\overline{\sigma_1},\pi_3,\overline{\pi_{4}}\right),\f_{2}\left(\pi_2,\overline{\sigma_2},\overline{\pi_{3}},\pi_{4}\right) \end{equation}
are both $(\lfloor \sqrt{n}/8 \rfloor,\tfrac{1}{4}n^{-r})$-uncaptureable. 

Next, we apply Lemma \ref{lem:uncap-very-quasi} to the pair (\ref{eq:boosted-pair}) with $N = \lfloor \sqrt{n}/8 \rfloor$, $c = 1/4$, $s=r+t$ and $\epsilon>0$ satisfying
$$\frac{(r \log n + 2 \log 2)(r+t)}{\log(1+\epsilon)} \leq \lfloor \sqrt{n}/8 \rfloor;$$
this holds provided
\begin{equation}\label{eq:epsilon-bound}
\epsilon = \Omega(r(r+t)n^{-1/2} \log n).
\end{equation}
(For concreteness, we may take $\epsilon = n^{-1/3}$, for example.) We obtain bijections $\pi_5,\pi_6$ bijecting between sets of size less than
$$(r \log n + 2\log 2)(r+t)/\log(1+\epsilon),$$
such that
$$\f_{1}\left(\pi_{1},\overline{\sigma_1},\pi_3,\overline{\pi_{4}},\pi_5,\overline{\pi_6}\right),\f_{2}\left(\pi_2,\overline{\sigma_2},\overline{\pi_{3}},\pi_{4},\overline{\pi_5},\pi_6\right)$$
are both $\left(r+t,1+2\epsilon\right)$-quasiregular, with measures greater than $\tfrac{1}{8}n^{-r}$.

Finally, we apply Proposition \ref{prop:extra-agreements} to the pair 
$$\f_{1}\left(\pi_{1},\overline{\sigma_1},\pi_3,\overline{\pi_{4}},\pi_5,\overline{\pi_6}\right),\f_{2}\left(\pi_2,\overline{\sigma_2},\overline{\pi_{3}},\pi_{4},\overline{\pi_5},\pi_6\right),$$
with $t-1-u$ in place of $t$ and $2\epsilon$ in place of $\epsilon$, producing a bijection $\pi_7$ between sets of size $t-1-u$, such that the domain of $\pi_7$ is disjoint from the domains of the $\pi_i$ (for $i \leq 6$) and of the $\sigma_i$ (for $i \leq 2$), the range of $\pi_7$ is disjoint from the ranges of the $\pi_i$ (for $i \leq 6$) and of the $\sigma_i$ (for $i \leq 2$),
\begin{align*}
\mu(\f_{1}\left(\pi_{1},\overline{\sigma_1},\pi_3,\overline{\pi_{4}},\pi_5,\overline{\pi_6},\pi_7\right)) & \geq (1-8\epsilon)\mu(\f_{1}\left(\pi_{1},\overline{\sigma_1},\pi_3,\overline{\pi_{4}},\pi_5,\overline{\pi_6}\right)),\\
\mu(\f_{2}\left(\pi_2,\overline{\sigma_2},\overline{\pi_{3}},\pi_{4},\overline{\pi_5},\pi_6,\pi_7\right)) & \geq (1-8\epsilon)\mu(\f_{2}\left(\pi_2,\overline{\sigma_2},\overline{\pi_{3}},\pi_{4},\overline{\pi_5},\pi_6\right)),
\end{align*}
and 
$\f_{1}\left(\pi_{1},\overline{\sigma_1},\pi_3,\overline{\pi_{4}},\pi_5,\overline{\pi_6},\pi_7\right),\ \f_{2}\left(\pi_2,\overline{\sigma_2},\overline{\pi_{3}},\pi_{4},\overline{\pi_5},\pi_6,\pi_7\right)$ are both $(r+u+1,1+16\epsilon)$-quasiregular. Now write $\rho_1 := \pi_1 \cup \pi_3 \cup \pi_5 \cup \pi_7$, $\rho_2 : = \pi_2 \cup \pi_4 \cup \pi_6 \cup \pi_7$, $\tau_1: = \sigma_1 \cup \pi_4 \cup \pi_6$ and $\tau_2 = \sigma_2 \cup \pi_3 \cup \pi_5$, so that
\begin{align} \label{eq:pi-defn} \f_{1}\left(\pi_{1},\overline{\sigma_1},\pi_3,\overline{\pi_{4}},\pi_5,\overline{\pi_6},\pi_7\right) & = \f_1(\rho_1,\overline{\tau_1}),\\
\f_{2}\left(\pi_2,\overline{\sigma_2},\overline{\pi_{3}},\pi_{4},\overline{\pi_5},\pi_6,\pi_7\right) & = \f_2(\rho_2,\overline{\tau_2}).\nonumber
\end{align}
By construction, $\f_1(\rho_1,\overline{\tau_1})$ and $\f_2(\rho_2,\overline{\tau_2})$ are both $(r+u+1,1+16\epsilon)$-quasiregular, with measures greater than $\tfrac{1}{9}n^{-r}$, and $\rho_1$ and $\rho_2$ agree with one another in exactly $t-1$ places. Moreover, we may assume that $\Domain(\rho_i) \cap \Domain(\tau_i) = \emptyset$ and $\Range(\rho_i) \cap \Range(\tau_i) = \emptyset$ for $i=1,2$, by simply deleting pairs $x \mapsto y$ in $\tau_i$ such that $x \in \Domain(\rho_i)$ or $y \in \Range(\rho_i)$. Finally, we have
\begin{equation}\label{eq:domain-bound} |\Domain(\rho_i)|+|\Domain(\tau_i)| = O(\sqrt{n}) \quad (i=1,2),\end{equation}
by construction.

Let $X$ and $Y$ be disjoint sets of size $n$, representing the domain and the range respectively of permutations in $S_n$. Fixing a bijection from $[n]$ to $X$ and a bijection from $[n]$ to $Y$, we may view permutations in $S_n$ as bijections from $X$ to $Y$, or as matchings from $X$ to $Y$. Similarly, may identify the bijections $\rho_1,\rho_2,\tau_1$ and $\tau_2$ with partial matchings $M_1,M_2,N_1,N_2$ (respectively) from $X$ to $Y$. Note that, by construction, $M_1,M_2,N_1$ and $N_2$ satisfy the following
\subsubsection*{`Good properties'.}
\begin{itemize}
\item $M_i$ is vertex-disjoint from $N_i$, for each $i \in \{1,2\}$.
\item $E(N_i) \subset E(M_{3-i})$ for each $i\in \{1,2$\}.
\item for each $i \in \{1,2\}$ and each edge $xy$ of $M_i$, exactly one of the following holds.
\begin{enumerate}
\item either $xy$ is an edge of $M_{3-i} \cup N_{3-i}$, or else
\item exactly one of $x$ and $y$ is incident to an edge of $M_{3-i}$ (in which case $xy$ is vertex-disjoint from $N_{3-i}$).
\end{enumerate}
\end{itemize}

Our first aim is to reduce to the case where all edges $xy$ of $M_i$ satisfy (1) above, for each $i \in \{1,2\}$, while (approximately) preserving the quasiregularity and high-measure properties. Let us say that edges $xy$ of $M_i$ that satisfy (1) above are of {\em type 1}, and that those satisfying (2) are of {\em type 2} (for $i = 1,2$). Since $M_i \cup N_i$ is a matching (for $i=1,2$), no edge of type 1 can be incident to an edge of type 2, and therefore the type-2 edges of $M_1 \cup M_2$ form a union of paths and cycles, where the paths have length at least 2 and the cycles have (even) length at least 4 (each cycle has even length, since its vertices alternate between $X$ and $Y$). (Here, and henceforth, the {\em length} of a path or cycle denotes its number of edges.) We will apply an algorithm that eliminates all of these paths and cycles of type-2 edges, step by step, at the cost of slightly reducing $n$ (viz., by $O(\sqrt{n})$). This algorithm will rely on the following three technical claims.

The first claim will enable us to eliminate all the cycles of type-2 edges from $M_1 \cup M_2$, one by one.

\begin{claim}
\label{claim:cycles}
Let $\eta >0$ and let $k,n,s,t \in \mathbb{N}$. Suppose that $\f_1(\rho_1,\overline{\tau_1}) \subset S_n(\rho_1,\overline{\tau_1})$ and $\f_2(\rho_2,\overline{\tau_2}) \subset S_n(\rho_2,\overline{\tau_2})$ are both $(s,1+\eta)$-quasiregular, and that $\rho_1$ and $\rho_2$ agree with one another in exactly $t-1$ places. Moreover, suppose that $\Domain(\rho_i) \cap \Domain(\tau_i) = \emptyset$ and $\Range(\rho_i) \cap \Range(\tau_i) = \emptyset$ for each $i\in \{1,2\}$. Let $X$ and $Y$ be disjoint sets of size $n$ (as above), let $M_1,M_2,N_1$ and $N_2$ be the matchings from $X$ to $Y$ associated with $\rho_1,\rho_2,\tau_1$ and $\tau_2$ respectively, and suppose that they satisfy the `good properties' above. Suppose that $M_1 \cup M_2$ contains a cycle of type-2 edges, of length $2k$. Then there exist a pair of bijections $\rho_1',\rho_2'$ between subsets of $[n-k]$ and a pair of bijections $\tau_1',\tau_2'$ between subsets of $[n-k]$, and a pair of families 
$$\f_1'(\rho_1',\overline{\tau_1'}) \subset S_{n-k}(\rho_1',\overline{\tau_1'}),\ \f_2'(\rho_2',\overline{\tau_2'}) \subset S_{n-k}(\rho_2',\overline{\tau_2'}),$$
such that
\begin{itemize}
\item $|\Domain(\rho_i')| = |\Domain(\rho_i)|-k$ for each $i\in \{1,2\}$;
\item $|\Domain(\tau_i')| = |\Domain(\tau_i)|$ for each $i \in \{1,2\}$;
\item $\Domain(\rho_i') \cap \Domain(\tau_i') = \emptyset$ and $\Range(\rho_i') \cap \Range(\tau_i') = \emptyset$, for each $i\in\{1,2\}$;
\item $\rho_1'$ and $\rho_2'$ agree wth one another in exactly $t-1$ places;
\item $\mu(\f_i'(\rho_i',\overline{\tau_i'})) = \mu(\f_i(\rho_i,\overline{\tau_i}))$ for each $i\in \{1,2\}$;
\item $\f_i'(\rho_i',\overline{\tau_i'})$ is $(s,1+\eta)$-quasiregular, for each $i\in\{1,2\}$;
\item The matchings $M_1',M_2',N_1',N_2'$ associated with $\rho_1',\rho_2',\tau_1',\tau_2'$ respectively satisfy the `good properties' above;
\item As an unlabelled graph, $M_1' \cup M_2'$ is obtained from $M_1 \cup M_2$ by the deletion of just one $(2k)$-cycle of type-2 edges;
\item If there exist two permutations
$$\sigma_1' \in \f_1'(\rho_1',\overline{\tau_1'}),\quad \sigma_2' \in \f_2'(\rho_2',\overline{\tau_2'})$$
that agree only on the $t-1$ points where $\rho_1'$ and $\rho_2'$ agree, then there exist two permutations
$$\sigma_1 \in \f_1(\rho_1,\overline{\tau_1}),\quad \sigma_2 \in \f_2(\rho_2,\overline{\tau_2})$$
that agree only on the $t-1$ points where $\rho_1$ and $\rho_2$ agree.
\end{itemize}
\end{claim}
\begin{proof}[Proof of claim.]
Let $C$ be a cycle of type-2 edges of $M_1 \cup M_2$, of length $2k$. Since $M_1$ and $M_2$ are both matchings, the edges of $C$ alternate between $M_1$ and $M_2$. Let $C = x_1y_1x_2y_2 \ldots x_k y_k x_1$, where $x_j \in X$ and $y_j \in Y$ for all $j \in [k]$. Without loss of generality, we may assume that $x_jy_j$ is an $M_1$-edge for each $j \in [k]$, and that $x_2 y_1$, $x_3y_2$, $\ldots$, $x_k y_{k-1}$ and $x_1 y_k$ are $M_2$-edges. It follows that $\rho_1(x_j) = y_j$ for all $j \in [k]$ and that $\rho_2(x_{j+1})=y_j$ for all $j \in [k]$ (where addition in the index is modulo $k$, i.e.\ $\rho_2(x_{1}) = y_k$). We now let $\theta \in S_{n}$ be the $k$-cycle defined by
$$\theta = (y_1\ y_2\ y_3 \ldots y_k),$$
and we consider the pair of families
$$\f_1(\rho_1,\overline{\tau_1}),\quad \theta (\f_2(\rho_2,\overline{\tau_2})) \subset S_{n}(\theta \rho_2,\overline{\tau_2}).$$
(We remark that the range of $\tau_2$ is disjoint from $\{y_1,\ldots,y_{k}\}$.) Note that any pair of permutations $\sigma_1 \in \f_1(\rho_1,\overline{\tau_1}),\ \sigma_2 \in \theta \f_2(\rho_2,\overline{\tau_2})$ map $x_j$ to $y_j$ for each $j \in [k]$, and in addition they agree on the $t-1$ points where $\rho_1$ and $\rho_2$ agree, but if they disagree everywhere else, then the two permutations
$$\sigma_1 \in \f_1(\rho_1,\overline{\tau_1}),\quad \theta^{-1} \sigma_2 \in \f_2(\rho_2,\overline{\tau_2})$$
agree only on the $t-1$ points where $\rho_1$ and $\rho_2$ agree. We now let $\rho_i''$ be the bijection obtained from $\rho_i$ by deleting $x_1,\ldots,x_k$ from the domain and deleting $y_1,\ldots,y_k$ from the range (for $i=1,2$), and we let
$$\f_1''(\rho_1'',\overline{\tau_1}),\quad \f_2''(\rho_2'',\overline{\tau_2})$$
be the families obtained from
$$\f_1(\rho_1,\overline{\tau_1}),\quad \theta (\f_2(\rho_2,\overline{\tau_2}))$$
respectively, by deleting $x_1,\ldots,x_k$ from the domain and deleting $y_1,\ldots,y_k$ from the range. The two families $\f_1''(\rho_1'',\overline{\tau_1}),\ \f_2''(\rho_2'',\overline{\tau_2})$ are of course families of bijections from $X \setminus \{x_1,\ldots,x_k\}$ to $Y \setminus \{y_1,\ldots,y_k\}$, but by `relabelling' (i.e., by applying fixed permutations to $X$ and $Y$), we can view them as (or replace them by) families of permutations in $S_{n-k}$. Let $\mathcal{F}_1'$ and $\mathcal{F}_2'$ be the families, and $\rho_1',\rho_2',\tau_1'$ and $\tau_2'$ the bijections, produced from $\mathcal{F}_1'',\mathcal{F}_2'',\rho_1'',\rho_2'',\tau_1$ and $\tau_2$ respectively, by this relabelling. By construction, if there exist two permutations
$$\sigma_1' \in \f_1'(\rho_1',\overline{\tau_1'}),\quad \sigma_2' \in \f_2'(\rho_2',\overline{\tau_2'})$$
that agree only on the $t-1$ points where $\rho_1'$ and $\rho_2'$ agree, then there exist two permutations
$$\sigma_1 \in \f_1(\rho_1,\overline{\tau_1}),\quad \sigma_2 \in \f_2(\rho_2,\overline{\tau_2})$$
that agree only on the $t-1$ points where $\rho_1$ and $\rho_2$ agree. Moreover, $\f_1'(\rho_1',\overline{\tau_1'})$ and $\f_2(\rho_2',\overline{\tau_2'})$ are both $(s,1+\eta)$-quasiregular, since $\f_1(\rho_1,\overline{\tau_1})$ and $\f_2(\rho_2,\overline{\tau_2})$ both are, and $\f_i'(\rho_i',\overline{\tau_i'})$ has the same measure as $\f_i(\rho_i,\overline{\tau_i})$ (for each $i\in \{1,2\}$). By construction, the matchings $M_1',M_2',N_1',N_2'$ associated with $\rho_1',\rho_2',\tau_1',\tau_2'$ respectively, satisfy the `good properties' above, and as an unlabelled graph, $M_1' \cup M_2'$ is obtained from $M_1 \cup M_2$ by the deletion of just one cycle of type-2 edges, having length $2k$.

\end{proof}

The next claim will enable us to eliminate all the even-length paths of type-2 edges from $M_1 \cup M_2$, one by one.

\begin{claim}
\label{claim:even-paths}
Let $\eta >0$ and let $k,n,s,t \in \mathbb{N}$. Suppose that $\f_1(\rho_1,\overline{\tau_1}) \subset S_n(\rho_1,\overline{\tau_1})$ and $\f_2(\rho_2,\overline{\tau_2}) \subset S_n(\rho_2,\overline{\tau_2})$ are both $(s,1+\eta)$-quasiregular, and that $\rho_1$ and $\rho_2$ agree with one another in exactly $t-1$ places. Moreover, suppose that $\Domain(\rho_i) \cap \Domain(\tau_i) = \emptyset$ and $\Range(\rho_i) \cap \Range(\tau_i) = \emptyset$ for each $i\in \{1,2\}$. Let $M_1,M_2,N_1$ and $N_2$ be the matchings from $X$ to $Y$ associated with $\rho_1,\rho_2,\tau_1$ and $\tau_2$ respectively, and suppose that they satisfy the `good properties' above. Suppose that $M_1 \cup M_2$ contains at least one path of type-2 edges, of length $2k$. Then there exist a pair of bijections $\rho_1',\rho_2'$ between subsets of $[n-k]$, a pair of bijections $\tau_1',\tau_2'$ between subsets of $[n-k]$, and a pair of families 
$$\f_1'(\rho_1',\overline{\tau_1'}) \subset S_{n-k}(\rho_1',\overline{\tau_1'}),\ \f_2'(\rho_2',\overline{\tau_2'}) \subset S_{n-k}(\rho_2',\overline{\tau_2'}),$$
such that
\begin{itemize}
\item $|\Domain(\rho_i')| = |\Domain(\rho_i)|-k$ for each $i\in \{1,2\}$;
\item $|\Domain(\tau_i')| = |\Domain(\tau_i)|$ for each $i \in \{1,2\}$;
\item $\Domain(\rho_i') \cap \Domain(\tau_i') = \emptyset$ and $\Range(\rho_i') \cap \Range(\tau_i') = \emptyset$, for each $i\in\{1,2\}$;
\item $\rho_1'$ and $\rho_2'$ agree with one another in exactly $t-1$ places;
\item $\mu(\f_i'(\rho_i',\overline{\tau_i'})) = \mu(\f_i(\rho_i,\overline{\tau_i}))$ for each $i\in \{1,2\}$;
\item $\f_i'(\rho_i',\overline{\tau_i'})$ is $(s,1+\eta)$-quasiregular, for each $i\in\{1,2\}$;
\item The matchings $M_1',M_2',N_1',N_2'$ associated with $\rho_1',\rho_2',\tau_1',\tau_2'$ respectively, satisfy the `good properties' above; 
\item As an unlabelled graph, $M_1' \cup M_2'$ is obtained from $M_1 \cup M_2$ by the deletion of just one path of type-2 edges of length $2k$;
\item If there exist two permutations
$$\sigma_1' \in \f_1'(\rho_1',\overline{\tau_1'}),\quad \sigma_2' \in \f_2'(\rho_2',\overline{\tau_2'})$$
that agree only on the $t-1$ points where $\rho_1'$ and $\rho_2'$ agree, then there exist two permutations
$$\sigma_1 \in \f_1(\rho_1,\overline{\tau_1}),\quad \sigma_2 \in \f_2(\rho_2,\overline{\tau_2})$$
that agree only on the $t-1$ points where $\rho_1$ and $\rho_2$ agree.
\end{itemize}
\end{claim}
\begin{proof}[Proof of claim.]
Let $P$ be a path of type-2 edges in $M_1 \cup M_2$, of length $2k$. As in the proof of Claim \ref{claim:cycles}, since $M_1$ and $M_2$ are both matchings, the edges of $P$ alternate between $M_1$ and $M_2$. Without loss of generality, we may assume that $P$ has both end-vertices in $X$. (The case where $P$ has both end-vertices in $Y$ can be dealt with by interchanging $X$ and $Y$.) Let $P = x_1y_1x_2y_2 \ldots x_{k}y_{k}x_{k+1}$. Without loss of generality, we may assume that $x_j y_j$ is an $M_1$-edge for all $j \in [k]$, and that $x_{j+1} y_{j}$ is an $M_2$-edge, for all $j \in [k]$. It follows that $\rho_1(x_j) = y_j$ for all $j \in [k]$ and that $\rho_2(x_{j+1})=y_j$ for all $j \in [k]$. We now let $\theta \in S_{n}$ be the $(k+1)$-cycle defined by
$$\theta = (x_1\ x_2\ x_3 \ldots x_{k+1}),$$
and we consider the pair of families
$$\f_1(\rho_1,\overline{\tau_1}),\quad \f_2(\rho_2,\overline{\tau_2})\theta \subset S_{n}(\rho_2 \theta,\overline{\tau_2}).$$
(We remark that the domain of $\tau_2$ is disjoint from $\{x_1,\ldots,x_{k+1}\}$.) Note that any pair of permutations $\sigma_1 \in \f_1(\rho_1,\overline{\tau_1}),\ \sigma_2 \in \f_2(\rho_2,\overline{\tau_2})\theta$ map $x_j$ to $y_j$ for each $j \in [k]$, and in addition they agree on the $t-1$ points where $\rho_1$ and $\rho_2$ agree, but if they disagree everywhere else, then the two permutations
$$\sigma_1 \in \f_1(\rho_1,\overline{\tau_1}),\quad \theta^{-1} \sigma_2 \in \f_2(\rho_2,\overline{\tau_2})$$
agree only on the $t-1$ points where $\rho_1$ and $\rho_2$ agree. We now let $\rho_i''$ be the bijection obtained from $\rho_i$ by deleting $x_1,\ldots,x_{k}$ from the domain and deleting $y_1,\ldots,y_{k}$ from the range (for $i=1,2$), and we let
$$\f_1''(\rho_1'',\overline{\tau_1}),\quad \f_2''(\rho_2'',\overline{\tau_2})$$
be the families obtained from
$$\f_1(\rho_1,\overline{\tau_1}),\quad \f_2(\rho_2,\overline{\tau_2})\theta$$
respectively, by deleting $x_1,\ldots,x_{k}$ from the domain and deleting $y_1,\ldots,y_{k}$ from the range. As in the proof of Claim \ref{claim:cycles}, the families $\f_1''(\rho_1'',\overline{\tau_1}),\ \f_2''(\rho_2'',\overline{\tau_2})$ are of course families of bijections from $X \setminus \{x_1,\ldots,x_k\}$ to $Y \setminus \{y_1,\ldots,y_k\}$, but by `relabelling' (i.e., by applying fixed permutations to $X$ and $Y$), we may view them as (or replace them by) families of permutations in $S_{n-k}$. Let $\mathcal{F}_1'$ and $\mathcal{F}_2'$ be the families, and $\rho_1',\rho_2',\tau_1'$ and $\tau_2'$ the bijections, produced from $\mathcal{F}_1'',\mathcal{F}_2'',\rho_1'',\rho_2'',\tau_1$ and $\tau_2$ respectively, by this relabelling. By construction, if there exist two permutations
$$\sigma_1' \in \f_1'(\rho_1',\overline{\tau_1'}),\quad \sigma_2' \in \f_2'(\rho_2',\overline{\tau_2'})$$
that agree only on the $t-1$ points where $\rho_1'$ and $\rho_2'$ agree, then there exist two permutations
$$\sigma_1 \in \f_1(\rho_1,\overline{\tau_1}),\quad \sigma_2 \in \f_2(\rho_2,\overline{\tau_2})$$
that agree only on the $t-1$ points where $\rho_1$ and $\rho_2$ agree. Moreover, $\f_1'(\rho_1',\overline{\tau_1'})$ and $\f_2(\rho_2',\overline{\tau_2'})$ are both $(s,1+\eta)$-quasiregular, since $\f_1(\rho_1,\overline{\tau_1})$ and $\f_2(\rho_2,\overline{\tau_2})$ both are, and $\f_i'(\rho_i',\overline{\tau_i'})$ has the same measure as $\f_i(\rho_i,\overline{\tau_i})$ (for each $i \in \{1,2\}$). By construction, the matchings $M_1',M_2',N_1',N_2'$ associated with $\rho_1',\rho_2',\tau_1',\tau_2'$ respectively, satisfy the `good properties' above, and as an unlabelled graph, $M_1' \cup M_2'$ is obtained from $M_1 \cup M_2$ by the deletion of just one path of type-2 edges, having length $2k$.
\end{proof}

Our final claim will enable us to eliminate all the odd-length paths of type-2 edges from $M_1 \cup M_2$. It will be more convenient for us to deal with all of these paths at once, rather than one by one, despite the slight notational complication.

\begin{claim}
\label{claim:odd-paths}
Let $0 < \eta \leq 1$ and let $b,K,Q,n,s,t \in \mathbb{N}$ with
\begin{equation}\label{eq:lower-bound-n} n \geq b+\max\{Q^2,16/\eta^2\}.
\end{equation}
Suppose that $\f_1(\rho_1,\overline{\tau_1}) \subset S_n(\rho_1,\overline{\tau_1})$ and $\f_2(\rho_2,\overline{\tau_2}) \subset S_n(\rho_2,\overline{\tau_2})$ are both $(s,1+\eta)$-quasiregular, and that $\rho_1$ and $\rho_2$ agree with one another in exactly $t-1$ places. Moreover, suppose that $|\Domain(\rho_i)|+|\Domain(\tau_i)| \leq b$ for each $i \in \{1,2\}$, and that $\Domain(\rho_i) \cap \Domain(\tau_i) = \emptyset$ and $\Range(\rho_i) \cap \Range(\tau_i) = \emptyset$ for each $i\in\{1,2\}$. Let $M_1,M_2,N_1$ and $N_2$ be the matchings from $X$ to $Y$ associated with $\rho_1,\rho_2,\tau_1$ and $\tau_2$ respectively, and suppose that they satisfy the `good properties' above. Suppose that $M_1 \cup M_2$ contains exactly $Q$ paths of type-2 edges having odd length, and that these have total length $2K+Q$. Then there exist a pair of bijections $\rho_1',\rho_2'$ between subsets of $[n-K]$, a pair of bijections $\tau_1',\tau_2'$ between subsets of $[n-K]$, and a pair of families 
$$\f_1'(\rho_1',\overline{\tau_1'}) \subset S_{n-K}(\rho_1',\overline{\tau_1'}),\ \f_2'(\rho_2',\overline{\tau_2'}) \subset S_{n-K}(\rho_2',\overline{\tau_2'}),$$
such that
\begin{itemize}
\item $|\Domain(\rho_i')| = |\Domain(\rho_i)|-k$ for each $i\in \{1,2\}$;
\item $|\Domain(\tau_i')| = |\Domain(\tau_i)|$ for each $i \in \{1,2\}$;
\item $\Domain(\rho_i') \cap \Domain(\tau_i') = \emptyset$ and $\Range(\rho_i') \cap \Range(\tau_i') = \emptyset$, for each $i\in\{1,2\}$;
\item $\rho_1'$ and $\rho_2'$ agree wth one another in exactly $t-1$ places;
\item $\mu(\f_i'(\rho_i',\overline{\tau_i'})) \geq \tfrac{1}{2}\mu(\f_i(\rho_i,\overline{\tau_i}))$ for each $i\in \{1,2\}$;
\item $\f_i'(\rho_i',\overline{\tau_i'})$ is $(s,1+3\eta)$-quasiregular, for each $i\in\{1,2\}$;
\item The matchings $M_1',M_2',N_1',N_2'$ associated with $\rho_1',\rho_2',\tau_1',\tau_2'$ respectively, satisfy the `good properties' above; 
\item As an unlabelled graph, $M_1' \cup M_2'$ is obtained from $M_1 \cup M_2$ by deleting all $Q$ of the paths of type-2 edges of odd length, and replacing each by a type-1 edge;
\item If there exist two permutations
$$\sigma_1' \in \f_1'(\rho_1',\overline{\tau_1'}),\quad \sigma_2' \in \f_2'(\rho_2',\overline{\tau_2'})$$
that agree only on the $t-1$ points where $\rho_1'$ and $\rho_2'$ agree, then there exist two permutations
$$\sigma_1 \in \f_1(\rho_1,\overline{\tau_1}),\quad \sigma_2 \in \f_2(\rho_2,\overline{\tau_2})$$
that agree only on the $t-1$ points where $\rho_1$ and $\rho_2$ agree.
\end{itemize}
\end{claim}

\begin{proof}[Proof of claim.]
Let $P_1,\ldots,P_Q$ denote the odd-length paths of type-2 edges in $M_1 \cup M_2$. As in the proof of Claim \ref{claim:cycles}, since $M_1$ and $M_2$ are both matchings, the edges of $P_q$ alternate between $M_1$ and $M_2$, for each $q \in [Q]$. Moreover, one end of $P_q$ must be in $X$ and the other must be in $Y$. Since $P_q$ has odd length, either it starts and ends with an edge of $M_1$, or else it starts and ends with an edge of $M_2$. For each $i \in \{1,2\}$, let $\mathcal{Q}_i$ denote the set of indices $q \in [Q]$ such that $P_q$ starts and ends with an edge of $M_i$. For each $q \in [Q]$, write $P_q = x_1^{(q)}y_1^{(q)}x_2^{(q)}y_2^{(q)} \ldots x_{k(q)}^{(q)}y_{k(q)}^{(q)}x_{k(q)+1}^{(q)}y_{k(q)+1}^{(q)}$, and let $\theta_q \in S_{n}$ be the $(k(q)+1)$-cycle defined by
$$\theta_q = (x_1^{(q)}\ x_2^{(q)}\ x_3^{(q)} \ldots x_{k(q)+1}^{(q)}).$$
We now let
$$\theta := \prod_{q \in \mathcal{Q}_1}\theta_q,\quad \theta': = \prod_{q \in \mathcal{Q}_2}\theta_q.$$
For each $i \in \{1,2\}$, we let $\tau_i^{*}$ be the bijection obtained from $\tau_i$ by adjoining $x_{k(q)+1}^{(q)} \mapsto y_{k(q)+1}^{(q)}$ to $\tau_i$, for each $q \in \mathcal{Q}_{3-i}$, and we let $\tau_i''$ be the bijection obtained from $\tau_i$ by adjoining $x_{1}^{(q)} \mapsto y_{k(q)+1}^{(q)}$ to $\tau_i$, for each $q \in \mathcal{Q}_{3-i}$. Consider the pair of families
$$\f_1(\rho_1,\overline{\tau_1''})\theta' \subset S_{n}(\rho_1 \theta',\overline{\tau_1^{*}}) ,\quad \f_2(\rho_2,\overline{\tau_2''})\theta \subset S_{n}(\rho_2 \theta, \overline{\tau_2^{*}}),$$
and let 
$$\f_1''(\rho_1'',\overline{\tau_1''}),\quad \f_2''(\rho_2'',\overline{\tau_2''})$$
be the families obtained from
$$\f_1(\rho_1,\overline{\tau_1''}) \theta',\quad \f_2(\rho_2,\overline{\tau_2''})\theta$$
respectively, by deleting $\{x_j^{(q)}:\ q \in [Q],\ j \in [k(q)]\}$ from the domain and deleting $\{y_j^{(q)}:\ q \in [Q],\ j \in [k(q)]\}$ from the range. Note that $K = \sum_{q=1}^{Q}k(q)$. The two families $\f_1''(\rho_1'',\overline{\tau_1''}),\quad \f_2''(\rho_2'',\overline{\tau_2''})$ are families of bijections from $X \setminus \{x_j^{(q)}:\ q \in [Q],\ j \in [k(q)]\}$ to $Y \setminus \{y_j^{(q)}:\ q \in [Q],\ j \in [k(q)]\}$, but by `relabelling' (i.e., by applying fixed permutations to $X$ and $Y$), we may view them as (or replace them by) families of permutations in $S_{n-K}$. Let $\mathcal{F}_1'$ and $\mathcal{F}_2'$ be the families, and $\rho_1',\rho_2',\tau_1'$ and $\tau_2'$ the bijections, produced from $\mathcal{F}_1'',\mathcal{F}_2'',\rho_1'',\rho_2'',\tau_1''$ and $\tau_2''$ respectively, by this relabelling. By construction, if there exist two permutations
$$\sigma_1' \in \f_1'(\rho_1',\overline{\tau_1'}),\quad \sigma_2' \in \f_2'(\rho_2',\overline{\tau_2'})$$
that agree only on the $t-1$ points where $\rho_1'$ and $\rho_2'$ agree, then there exist two permutations
$$\sigma_1 \in \f_1(\rho_1,\overline{\tau_1}),\quad \sigma_2 \in \f_2(\rho_2,\overline{\tau_2})$$
that agree only on the $t-1$ points where $\rho_1$ and $\rho_2$ agree. Moreover, by Claim \ref{claim:quasi-small-error} and our assumption (\ref{eq:lower-bound-n}) on $n$, it follows that $\f_1'(\rho_1',\overline{\tau_1'})$ and $\f_2'(\rho_2',\overline{\tau_2'})$ are both $(s,1+3\eta)$-quasiregular and that $\mu(\f_i'(\rho_i',\overline{\tau_i'})) \geq \tfrac{1}{2}\mu(\f_i(\rho_i,\overline{\tau_i}))$ for each $i\in \{1,2\}$. By construction, the matchings $M_1',M_2',N_1',N_2'$ associated with $\rho_1',\rho_2',\tau_1',\tau_2'$ respectively, satisfy the `good properties' above, and as an unlabelled graph, $M_1' \cup M_2'$ is obtained from $M_1 \cup M_2$ by replacing each odd-length path of type-2 edges by a type-1 edge.

\end{proof}

Armed with these three claims, we now describe our algorithm. Starting with the pair of families
 $$\mathcal{F}_1(\rho_1,\overline{\tau_1}) \subset S_n(\rho_1,\overline{\tau_1}),\quad \mathcal{F}_2(\rho_2,\overline{\tau_2}) \subset S_n(\rho_2,\overline{\tau_2})$$
defined in (\ref{eq:pi-defn}), we first apply Claim \ref{claim:cycles} (with $s=r+u+1$ and $\eta = 16\epsilon$) to each cycle of type-2 edges in succession, and then we apply Claim \ref{claim:even-paths} (with $s=r+u+1$ and $\eta = 16\epsilon$) to each even-length path of type-2 edges in succession (reducing $n$ by at most $O(\sqrt{n})$ after all of these applications, by virtue of (\ref{eq:domain-bound})). Abusing notation slightly (to avoid clutter), let
\begin{equation}\label{eq:second-pair} \mathcal{F}_1(\rho_1,\overline{\tau_1}) \subset S_{\tilde{n}}(\rho_1,\overline{\tau_1}),\quad \mathcal{F}_2(\rho_2,\overline{\tau_2}) \subset S_{\tilde{n}}(\rho_2,\overline{\tau_2})
\end{equation}
be the pair of families produced by this process, and let $M_1,M_2,N_1,N_2$ be the corresponding matchings. Then $n-\tilde{n} = O(\sqrt{n})$, the matchings $M_1,M_2,N_1,N_2$ still satisfy the `good properties' above, and $M_1 \cup M_2$ contains no cycle of type-2 edges and no even-length path of type-2 edges. Further, $\f_1(\rho_1,\overline{\tau_1})$ and $\f_2(\rho_2,\overline{\tau_2})$ are both $(r+u+1,1+16\epsilon)$-quasiregular, with measures greater than $\tfrac{1}{9}n^{-r}$, and $\rho_1$ and $\rho_2$ agree with one another in exactly $t-1$ places.

We can now apply Claim \ref{claim:odd-paths} to the pair of families (\ref{eq:second-pair}), with $s = r+u+1$, with $\eta = 16\epsilon$, with $\tilde{n} = n-O(\sqrt{n})$ in place of $n$, and with $\max\{b,K,Q\} = O(\sqrt{n})$ (again, by virtue of (\ref{eq:domain-bound})). Note that, by our choice of $\epsilon = n^{-1/3}$, the hypothesis 
$$\tilde{n} \geq b+\max\{Q^2,16/\eta^2\}$$
is indeed satisfied (provided $n$ is sufficiently large depending on $r$ and $t$). Writing $n' = \tilde{n}-K = n-O(\sqrt{n})$, we obtain a pair of bijections
$\rho_1',\rho_2'$ between subsets of $[n']$, a pair of bijections $\tau_1',\tau_2'$ between subsets of $[n']$, and a pair of families 
$$\f_1'(\rho_1',\overline{\tau_1'}) \subset S_{n'}(\rho_1',\overline{\tau_1'}),\ \f_2'(\rho_2',\overline{\tau_2'}) \subset S_{n'}(\rho_2',\overline{\tau_2'}),$$
which are both $(r+u+1,1+48\epsilon)$-quasiregular with measures greater than $\tfrac{1}{18}n^{-r} > \tfrac{1}{19}(n')^{-r}$ (the last inequality using the fact that $n'-n = O(\sqrt{n})$ and that $n$ is sufficiently large depending on $r$). Observe that $\rho_1'$ and $\rho_2'$ agree with one another at exactly $t-1$ places, that the matchings $M_1',M_2',N_1',N_2'$ associated to $\rho_1', \rho_2',\tau_1',\tau_2'$ respectively satisfy the `good properties' above, and in addition that $M_1' \cup M_2'$ consists only of type-1 edges. 

For notational convenience, we now write
$$\f_1'(\rho_1',\overline{\tau_1'}) =: \g_1(\theta_1,\overline{\phi_2}),\quad \f_2'(\rho_2',\overline{\tau_2'}) =:\g_2(\theta_2,\overline{\phi_1}).$$
Since $M_1' \cup M_2'$ consists only of type-1 edges, we can rewrite the pair
$$\g_1(\theta_1,\overline{\phi_2}), \quad \g_2(\theta_2, \overline{\phi_1})$$
as
$$\g_1(\psi \cup \phi_1,\overline{\phi_2}), \quad \g_2(\psi \cup \phi_2, \overline{\phi_1}),$$
where $\psi$ is the restriction of $\theta_1$ to the set of points where $\theta_1$ and $\theta_2$ agree, so that $|\Domain(\psi)|=t-1$. Observe that if two permutations $\sigma_1 \in \g_1(\psi \cup \phi_1,\overline{\phi_2})$ and $\sigma_2 \in \g_2(\psi \cup \phi_2, \overline{\phi_1})$ disagreed everywhere outside $\Domain(\psi)$, then there would exist two permutations in $\f$ that agree only where $\rho_1$ and $\rho_2$ agree (so in exactly $t-1$ places); this would contradict our assumption that $\f$ is $(t-1)$-intersection-free.

Given our pair of families
$$\g_1(\psi \cup \phi_1,\overline{\phi_2}) \subset S_{n'}(\psi \cup \phi_1,\overline{\phi_2}),\quad \g_2(\psi\cup \phi_2, \overline{\phi_1}) \subset S_{n'}(\psi \cup \phi_2,\overline{\phi_1}),$$
which are $(r+u+1,1+48\epsilon)$-quasiregular with measures greater than $\tfrac{1}{19}(n')^{-r}$, we observe that $S_{n'}(\psi)$ can be viewed as a copy of $S_{n'-t+1}$ (recall that $|\Domain(\psi)|=t-1$). Under this identification, writing $n'': = n'-t+1$, we obtain a pair of families
$$\g_1'(\phi_1,\overline{\phi_2}) \subset S_{n''}(\phi_1,\overline{\phi_2}),\ \g_2'(\phi_2,\overline{\phi_1}) \subset S_{n''}(\phi_2,\overline{\phi_1})$$
simply by deleting the domain and range of $\psi$ from each permutation in $\g_1(\psi \cup\phi_1,\overline{\phi_2})$ and $\g_2(\psi\cup \phi_2, \overline{\phi_1})$ respectively, and relabelling the ground set as $[n'']$ if necessary. Clearly, these new families are also $(r+u+1,1+48\epsilon)$-quasiregular, with measures greater than $\tfrac{1}{19}(n')^{-r} > \tfrac{1}{20}(n'')^{-r}$. Moreover, if two permutations $\sigma_1 \in \g_1'(\phi_1,\overline{\phi_2})$ and $\sigma_2 \in \g_2'(\phi_2, \overline{\phi_1})$ disagreed everywhere, then there would exist two permutations in $\f$ that agree only where $\rho_1$ and $\rho_2$ agree (so in exactly $t-1$ places); this would contradict our assumption that $\f$ is $(t-1)$-intersection-free.

The last step is to reduce to the case where $\phi_1$ and $\phi_2$ are empty; this may involve replacing $\mathcal{G}_1'$ and $\mathcal{G}_2'$ by fractional families, i.e.\ $[0,1]$-valued functions on $S_{\hat{n}}$, for some $\hat{n} \leq n''$. To this end, we define a generalisation of the `fixing' operators introduced by Cameron and Ku \cite{ck}. 
\begin{defn}
If $\pi_1:S_1\to T_1$ and $\pi_2:S_2 \to T_2$ are bijections with disjoint domains and disjoint ranges, let $\mathcal{S}_{\pi_{1},\overline{\pi_{2}}}\colon S_{n}\left(\pi_{1},\overline{\pi_{2}}\right)\to S_{n}\left(\pi_{1},\pi_{2}\right)$
be the operator defined as follows. If $S_2 = \{i_1,\ldots,i_\ell\}$ with $i_1 < \ldots < i_{\ell}$, and $\pi_2(i_k) = j_k$ for all $k \in [\ell]$, then 
$$\mathcal{S}_{\pi_{1},\overline{\pi_{2}}} := F_{\ell} \circ F_{\ell-1} \circ \ldots \circ F_{2} \circ F_1,$$
where $F_k(\sigma): = (j_k \ \sigma(i_k))\sigma$, for all $k \in [\ell]$ and all $\sigma \in S_n(\pi_1,\overline{\pi_2})$.
\end{defn} 
It is easy to check that the operators $F_k$ pairwise commute, so the restriction $i_1 < \ldots < i_{\ell}$ in the above definition is in fact redundant. We remark that the $\mathcal{S}_{\pi_{1},\overline{\pi_{2}}}$'s are a generalisation of the `fixing' operators introduced by Cameron and Ku in \cite{ck}: when $j_k = i_k$ for all $k$ and $\pi_1 = \emptyset$, we have $\mathcal{S}_{\pi_1,\overline{\pi_2}}(\sigma) = g_{i_1,\ldots,i_\ell}(\sigma)$ for all $\sigma \in S_n(\overline{\pi_2})$, using their notation.

We are now ready to define the function $f_{\mathcal{F},\pi_{1},\overline{\pi_{2}}}.$ 
\begin{defn}
Let $\f(\pi_1,\overline{\pi_2}) \subset S_n(\pi_1,\overline{\pi_2})$, where $\pi_1,\pi_2$ are as above. We define $\tilde{f}_{\f,\pi_{1},\overline{\pi_{2}}}\colon S_{n}\left(\pi_{1},\pi_{2}\right)\to\left[0,1\right]$
to be the function that sends a permutation $\tau\in S_{n}\left(\pi_{1},\pi_{2}\right)$
to the proportion of permutations in $\mathcal{S}_{\pi_{1},\overline{\pi_{2}}}^{-1}\left(\tau\right)$
that belong to the family $\f$. Equivalently,
$$\tilde{f}_{\f,\pi_{1},\overline{\pi_{2}}}(\tau) = \frac{|\s_{\pi_{1},\overline{\pi_{2}}}^{-1}\left(\tau\right)\cap \f|}{|\s_{\pi_{1},\overline{\pi_{2}}}^{-1}\left(\tau\right)|}\quad \forall \tau \in S_n(\pi_1,\pi_2).$$
\end{defn}

Observe that $\mathbb{E}[\tilde{f}_{\f,\pi_{1},\overline{\pi_{2}}}] = \mu(\f)$ for any family $\f \subset S_n(\pi_1,\overline{\pi_2})$, since $|\mathcal{S}_{\pi_{1},\overline{\pi_{2}}}^{-1}\left(\tau\right)|$ is the same for all permutations $\tau \in S_n(\pi_1,\pi_2)$.

We need the following easy observation.
\begin{observation}
\label{observation:swapping}
Let $s \in \mathbb{N}$ and $\alpha \geq 1$, and let $\pi_1,\pi_2$ be bijections with disjoint domains and disjoint ranges. Let $\f(\pi_1,\overline{\pi_2}) \subset S_n(\pi_1,\overline{\pi_2})$ and let $\g(\overline{\pi_1},\pi_2) \subset S_n(\overline{\pi_1},\pi_2)$. If every permutation in $\f(\pi_1,\overline{\pi_2})$ intersects every permutation in $\g(\overline{\pi_1},\pi_2)$, then the functions $f = \tilde{f}_{\f,\pi_{1},\overline{\pi_{2}}}:S_n(\pi_1,\pi_2) \to [0,1]$ and $g = \tilde{f}_{\g,\overline{\pi_{1}},\pi_{2}}:S_n(\pi_1,\pi_2) \to [0,1]$ satisfy $f(\sigma_1)g(\sigma_2)=0$ for all permutations $\sigma_1,\sigma_2 \in S_n(\pi_1,\pi_2)$ such that $\sigma_1$ and $\sigma_2$ intersect only on $\Domain(\pi_1) \cup \Domain(\pi_2)$. Moreover, if $\f(\pi_1,\overline{\pi_2})$ and $\g(\overline{\pi_1},\pi_2)$ are both $(s,\alpha)$-quasiregular, then so are $f$ and $g$.
\end{observation}

We now return to the proof. Let
\begin{align*} f'' & = \tilde{f}_{\g_1',\phi_{1},\overline{\phi_{2}}}\colon S_{n''}\left(\phi_{1},\phi_{2}\right)\to\left[0,1\right],\\
g'' & = \tilde{g}_{\g_2',\phi_{2},\overline{\phi_{1}}}\colon S_{n''}\left(\phi_{1},\phi_{2}\right)\to\left[0,1\right].
\end{align*}
Noting that $S_{n''}\left(\phi_{1},\phi_{2}\right)$ can be viewed as a copy of $S_{\hat{n}}$, where $\hat{n}:=n''-|\Domain(\phi_1)|-|\Domain(\phi_2)| = n-O(\sqrt{n})$, we produce from $f''$ and $g''$ two functions $f,g: S_{\hat{n}} \to [0,1]$ by deleting the domains and ranges of $\phi_1$ and $\phi_2$ and relabelling the ground set as $[\hat{n}]$ if necessary. It follows from Observation \ref{observation:swapping} that $f$ and $g$ are $(r+u+1,1+48\epsilon)$-quasiregular; clearly, they have expectations greater than $\tfrac{1}{20}(n'')^{-r} > \tfrac{1}{21}(\hat{n})^{-r}$. Finally, for any two permutations $\sigma_1,\sigma_2 \in S_{\hat{n}}$ disagreeing everywhere, we have $f(\sigma_1)g(\sigma_2)=0$.

Since $f$ and $g$ are $(r+u+1,1+48\epsilon)$-quasiregular, they are $(r+u+1,144\epsilon)$-quasirandom (and therefore $(r,144\epsilon)$-quasirandom), by Lemma \ref{lem:quasireg-quasirandom}. Applying Lemma \ref{lem:fourier} with $f_1=f$ and $f_2=g$ now yields a contradiction, proving Theorem \ref{thm:junta-approximation}.

\section{Proof of Theorem \ref{thm:extremal-forbidden}}
\label{sec:deduction}
In this section, we deduce Theorem \ref{thm:extremal-forbidden} from our junta approximation theorem, via a purely combinatorial argument.
\begin{proof}[Proof of Theorem \ref{thm:extremal-forbidden}]
Let $n \geq n_0$, where $n_0 = n_0(t) \in \mathbb{N}$ is to be chosen later. Let $\f \subset S_n$ be $(t-1)$-intersection-free with $|\f| = (n-t)!$. Let $r = r(t) \in \mathbb{N}$ to be chosen later (in fact, we may choose $r = t+1$). By Theorem \ref{thm:junta-approximation}, there exists a $t$-intersecting $C$-junta $\j \subset S_n$ such that $|\f \setminus \j| \leq Cn!/n^r$, where $C = C(r,t) \in \mathbb{N}$. Let $\j = \langle \pi_1,\ldots,\pi_l\rangle$, where $\pi_i :S_i \to T_i$ is a bijection for all $i \in [l]$ and $S_i,T_i \subset [n]$ for all $i \in [l]$. Note that $l \leq C$ and that $|S_i| \leq C$ for all $i \in [l]$. Since $\j$ is $t$-intersecting, provided $n_0 \geq t+1$ we must have $|\Domain(\pi_i)| \geq t$ for all $i \in [l]$. Suppose for a contradiction that $|\Domain(\pi_i)| \geq t+1$ for all $i \in [l]$. Then
$$|\j| \leq \sum_{i=1}^{l} |\langle \pi_i \rangle| \leq l(n-t-1)! \leq C(n-t-1)!,$$
so $|\f \setminus \j| \geq |\f|-|\j| \geq (n-t)!- C(n-t-1)! > (n-t)!/2 > Cn!/n^r$ (provided $r \geq t+1$ and $n_0$ is sufficiently large depending on $t$), a contradiction. Therefore, there exists $i \in [l]$ such that $|\Domain(\pi_i)| = t$; without loss of generality, we may assume that $|\Domain(\pi_1)| = t$. Then, since $\j$ is $t$-intersecting, provided $n_0 \geq t+1$ we must have $\pi_i = \pi_1$ for all $i \in [l]$, and therefore $\j = \langle \pi_1 \rangle$. By considering $\tau_1 \f \tau_2$ for appropriate $\tau_1,\tau_2 \in S_n$, we may assume that $\pi_1$ is the identity on $[t]$, i.e.\ that $\j = \{\sigma \in S_n:\ \sigma(i)=i\ \forall i \in [t]\}$. Now suppose, for a contradiction, that $\f \neq \j$. Then there exists some permutation $\rho \in \f \setminus \j$. Let $s$ be the number of fixed points of $\rho$ in $[t]$; then $0 \leq s \leq t-1$. Let $V =\rho^{-1}([t]) \setminus [t]$ and let $v = |V|$; note that $v \leq t$. We claim that the number of permutations in $\j$ that agree with $\rho$ in exactly $t-1$ places is at least
$$N: = {n-t-v \choose t-s-1}(n-2t+s+1)_{(v)}d_{n-2t+s-v+1}.$$
Indeed, to choose such a permutation $\sigma \in \j$, we can first choose a set $U$ of $t-s-1$ points of $\{t+1,\ldots,n\} \setminus V$ on which $\sigma$ agrees pointwise with $\rho$, we can then choose the images of the points in $V$ arbitrarily (there are $(n-2t+s+1)(n-2t+s-1)\ldots(n-2t+s-v+2) = (n-2t+s+1)_{(v)}$ such choices), and finally we can choose the images of the points in $\{t+1,\ldots,n\} \setminus (U \cup V)$ in such a way that $\sigma$ disagrees with $\rho$ everywhere on $\{t+1,\ldots,n\} \setminus (U \cup V)$ (there are $d_{n-2t+s-v+1}$ such choices, where as before, $d_k$ denotes the number of derangements of a $k$-element set). Since $\f$ is $(t-1)$-intersection-free, none of these permutations can be in $\f$, so $|\j \setminus \f| \geq N$. We have
\begin{align*}  N & \geq \left(\frac{n-2t}{t-s-1}\right)^{t-s-1} (n-2t+s+v+2)^{v} (1/e+o(1))(n-2t+s-v+1)!\\
& \geq \left(\frac{n-2t}{t-1}\right)^{t-s-1} (n-2t+2)^{v} (1/e+o(1))(n-2t+s-v+1)!\\
& = \Omega_t(1) n^{t-s-1} n^{v} n! / n^{2t-s+v-1}\\
& =  \Omega_t(1) n! / n^{t}.\end{align*}
Hence, $|\j \setminus \f| =\Theta_t(1) n! / n^{t}$. Since $|\f| = (n-t)! = |\j|$, it follows that $|\f \setminus \j| = |\j \setminus \f| =\Theta_t(1) n! / n^{t} > Cn!/n^r$ (provided $r \geq t+1$ and $n_0$ is sufficiently large depending on $t$), a contradiction.
\end{proof}

We now use Theorem \ref{thm:junta-approximation} to deduce Theorem \ref{thm:stability}, our 1\% stability result for families of permutations with a forbidden intersection.
\begin{proof}[Proof of Theorem \ref{thm:stability}]
This is very similar to the first part of the proof of Theorem \ref{thm:extremal-forbidden}. Let $K = K(r,t) >0$ to be chosen later. By an appropriate choice of $K$, we may assume that $n \geq t+1$. Let $\f \subset S_n$ be $(t-1)$-intersection-free with $|\f| \geq K(n-t-1)!$. By Theorem \ref{thm:junta-approximation}, there exists a $t$-intersecting $C$-junta $\j \subset S_n$ such that $|\f \setminus \j| \leq Cn!/n^r$, where $C = C(r,t) \in \mathbb{N}$. Let $\j = \langle \pi_1,\ldots,\pi_l\rangle$, where $\pi_i :S_i \to T_i$ is a bijection for all $i \in [l]$ and $S_i,T_i \subset [n]$ for all $i \in [l]$. Since $\j$ is $t$-intersecting and $n \geq t+1$, we must have $|\Domain(\pi_i)| \geq t$ for all $i \in [l]$. Suppose for a contradiction that $|\Domain(\pi_i)| \geq t+1$ for all $i \in [l]$. Then
$$|\f \setminus \j| \geq |\f| - |\j| \geq K(n-t-1)!- \sum_{i=1}^{l}|\langle \pi_i \rangle| \geq K(n-t-1)! - C(n-t-1)! \geq C(n-t-1)! > Cn!/n^{t+1} \geq Cn!/n^r$$
provided $K \geq 2C$, a contradiction. Therefore, there exists $i \in [l]$ such that $|\Domain(\pi_i)| = t$; without loss of generality, we may assume that $|\Domain(\pi_1)| = t$. Then, since $\j$ is $t$-intersecting and $n \geq t+1$, we must have $\pi_i = \pi_1$ for all $i \in [l]$, and therefore $\j = \langle \pi_1 \rangle$. We have
$$|\f \setminus \langle \pi_1 \rangle| \leq Cn!/n^r \leq K(n-r)!$$
provided $K \geq C$, proving the theorem.
\end{proof}

\section{Conclusion and open problems}
\label{sec:conc}
Two natural open problems are to determine, for each pair of integers $n$ and $t$, the largest possible $t$-intersecting (respectively $(t-1)$-intersection-free) families of permutations in $S_n$. The following conjecture in \cite{efp}, as to the former, remains open.
\begin{conj}
\label{conj:2t}
For any $n,t \in \mathbb{N}$, a maximum-sized \(t\)-intersecting family in \(S_{n}\) must be a double translate of one of the families
\[\mathcal{F}_{i} := \{\sigma \in S_{n}:\ \sigma \textrm{ has at least } t+i \textrm{ fixed points in } [t+2i]\}\ (0 \leq i \leq (n-t)/2),\]
i.e.\ it must be of the form \(\pi \mathcal{F}_i \tau\), for some \(\pi,\tau \in S_n\).
\end{conj}
In fact, even the following conjecture remains open.
\begin{conj}
\label{conj:large}
For any \(n,t \in \mathbb{N}\) with \(n > 2t\), the maximum-sized \(t\)-intersecting families in \(S_n\) are the \(t\)-stars.
\end{conj}
The arguments in this paper imply the conclusion of Conjecture \ref{conj:2t} only for $n \geq \exp(Ct \log t)$, for some absolute constant $C>0$. (On the other hand, the arguments in \cite{efp} require $n$ to be doubly exponential in $t$.) We believe that entirely new techniques will be required to deal with the case where $n = \Theta(t)$.

We note that, when $t$ is comparable to $n$, the largest $(t-1)$-intersection-free subfamilies of $S_n$ are very much larger than the largest $t$-intersecting subfamilies of $S_n$. Indeed, for any $t \leq n \in \mathbb{N}$, a $t$-intersecting subfamily of $S_n$ trivially has size at most ${n \choose t} (n-t)!$, the latter being the number of permutations that $t$-intersect a fixed permutation. On the other hand, a $(t-1)$-intersection-free subfamily of $S_n$ is precisely an independent set in the Cayley graph on $S_n$ generated by all the permutations with exactly $t-1$ fixed points; since this graph is regular of degree ${n \choose t-1}d_{n-t+1}$, by Tur\'an's theorem it has an independent set of size at least
$$\frac{n!}{{n \choose t-1}d_{n-t+1}+1},$$
which is $\gg {n \choose t} (n-t)!$ provided $t \geq (1+\delta)n/2$, $\delta >0$ and $n$ is sufficiently large depending on $\delta$. At the extreme, when $t=n-1$, the largest $(t-1)$-intersection-free subfamilies of $S_n$ are precisely the maximum-sized independent sets in the Cayley graph generated by the transpositions, i.e.\ they are precisely $A_n$ and $S_n \setminus A_n$, whereas the largest $t$-intersecting subfamilies are precisely the singletons. It would be interesting to determine the value of $n$ (for each $t$) for which the largest $(t-1)$-intersection-free subfamilies of $S_n$ are the same as the largest $t$-intersecting subfamilies.

\subsubsection*{Acknowledgements}
The authors would like to thank Eoin Long for helpful discussions.

\end{document}